\documentclass{mcom-l}

\usepackage{amssymb}

\usepackage{graphicx}

\usepackage[cmtip,all]{xy}
\usepackage{amsmath}   

\usepackage{mathtools}

\theoremstyle{plain}
\newtheorem{thm}{Theorem}[section]

\newtheorem{prop}[thm]{Proposition}

\newtheorem{assump}[thm]{Assumption}

\newtheorem{Lemma}[thm]{Lemma} 
\theoremstyle{definition}
\newtheorem{defn}[thm]{Definition}
\newtheorem{rem}[thm]{Remark}
\usepackage{graphicx}
\usepackage{float}       
\usepackage{booktabs}    
\usepackage{multirow}    
\usepackage{xcolor}
\usepackage{subcaption}  
\usepackage{comment}
\usepackage{algorithm}
\usepackage{algorithmicx}
\usepackage{algpseudocode}
\usepackage{listings}    

\theoremstyle{remark}
\usepackage{hyperref}
\hypersetup{
    colorlinks=true,
    linkcolor=blue,
    citecolor=blue,
    urlcolor=blue
}
\usepackage{cleveref} 
\usepackage{enumitem}
\setlist[enumerate,1]{label=(\roman*)} 
\usepackage{url}
\usepackage{comment}
\numberwithin{equation}{section}
\numberwithin{equation}{section}
\numberwithin{algorithm}{section}
\usepackage{bm}

\newcommand{\R}{\mathbb{R}} 

\newcommand{\dom}{\text{dom}}

\newcommand{\E}{\mathbb{E}}

\newcommand{\Expect}{\mathbb{E}}

\newcommand{\prox}{\text{prox}}

\newcommand{\env}{\operatorname{env}}

\newcommand{\norm}[1]{\left\| #1 \right\|}

\newcommand{\romannum}[1]{\uppercase\expandafter{\romannumeral #1\relax}}

\usepackage{amssymb}
\usepackage{amsmath}

\begin{document}

\title{A Stochastic Implicit  Proximal Point Algorithm for Solving Linearly Constrained Stochastic Minimax Problems}
\author{Kehan Zhu}
\address{School of Mathematical Sciences, Beijing University of Posts and Telecommunications, Beijing, China}
\email{zhukehan@bupt.edu.cn}

\author{Jiani Wang}
\address{School of Mathematical Sciences \& Key Laboratory of Mathematics and Information Networks, Beijing University of Posts and Telecommunications, Ministry of Education, Beijing, China}
\email{wjiani@bupt.edu.cn}
\thanks{This work was supported by the National Natural Science Foundation of China (No. 12401400), Scientific Research Startup Fund of Beijing University of Posts and Telecommunications (No. 510224054) and the State Key Laboratory of Scientific and Engineering Computing, Chinese Academy of Sciences.}

\author{Yu-Hong Dai}
\address{LSEC, ICMSEC, Academy of Mathematics and Systems Science, Chinese Academy of Sciences, Beijing, China}
\email{dyh@lsec.cc.ac.cn}
\thanks{This work was supported by National Key R\&D Program of China (2022YFA1004000), the National Natural Science Foundation of China (No. 12331011).}

\subjclass[2010]{Primary 65K05, 90C30; Secondary 90C06}

\date{\today}


\begin{abstract}
This paper presents a novel approach to solving large-scale minimax problems with nonsmooth regularizers. We propose a stochastic implicit  proximal point algorithm with variance reduction techniques where stochastic oracles are selected in two cases---with or without replacement. The semismooth Newton methods with Armijo line search is used to solve the implicit proximal point update subproblem in each iteration. The algorithm efficiently handles the strongly-convex-strongly-concave objective function with nonsmooth regularizers and coupling linear equations, which is proved to exhibit global q-linear convergence of the iterations to the saddle point and global r-linearly convergence of the multipliers to the multiplier set in expectation. Numerical experiments on machine learning problems demonstrate the superiority of the proposed method over state-of-the-art algorithms in terms of both computational efficiency and selection of the step sizes.
\end{abstract}
\keywords{
minimax problem; semismooth newton; stochastic proximal point method; variance reduction}

\maketitle


\section{Introduction}
\label{sec:intro}

\subsection{Problem setting}
Linearly constrained optimization problems arise in numerous scientific and 
engineering applications including signal processing \cite{lai2013augmented} and distributed optimization \cite{ling2015dlm}.
 Many studies\cite{han2018linear,powell1993number}, have developed the optimality theory and numerical algorithms for
  the linear constrained minimization problem, which is one of the most important constraint problems.
   In this paper, we consider the following linear constrained minimax problem 
\begin{equation}
    \begin{aligned}
    &\min_{x \in \Re^n} \max_{y \in \Re^m} \quad \varphi(x)+g(x)+f(x,y)-h(y)-\psi(y)\\
    &\text{subject}  ~\text{to} \quad Ax+By+c=0,
    \end{aligned}
    \label{problem1}
    \end{equation}
where
\[g(x)=\frac{1}{N} \sum_{i=1}^{N}g_i(x),\quad h(y)=\frac{1}{N} \sum_{i=1}^{N}h_i(y),\quad f(x,y)=\frac{1}{N}\sum_{i=1}^{N}f_i(x,y)\]
and  $g_i:\Re^n \to \Re,h_i:\Re^m \to \Re,f_i:\Re^n \times \Re^m \to \Re $ are  continuously differentiable and smooth functions, 
and $\varphi:\Re^n \to \Re,\psi:\Re^m \to \Re$ are extended real-valued proper lower semicontinuous convex functions. 
$A \in \Re ^{q \times n}, B \in \Re ^{q\times m},$ and $~c \in \Re^q.$

\subsection{Applications}

\subsubsection{Adversarial attacks in network flow problems}\label{subsec:adversarial_attacks_network_flow}
In a flow network, a user routes flow $\mathbf{x}$ from source to sink at minimum cost, while an adversary injects flow $\mathbf{y}$ to maximize the user's expected cost under uncertain edge costs. The stochastic minimax formulation is
\[
\max_{\mathbf{y} \in \mathcal{Y}} \min_{\mathbf{x} \in \mathcal{X}} \ \mathbb{E}[q(\mathbf{x}, \mathbf{y})],
\]
where $\mathcal{X}$ and $\mathcal{Y}$ are flow polyhedra with capacity and conservation constraints. This extends deterministic network interdiction models \cite{smith2013modern,fu2019network,salmeron2004analysis} to stochastic settings, applicable to communication or power networks with cost fluctuations.

\subsubsection{Linear regression}

We focus on the linear regression problem with smoothed $L_1$ regularization. Following the derivation in \cite{du2019linear}, this problem is fundamentally reformulated as a convex-concave saddle point problem. 

The primal problem is given by $\min_{x} \frac{1}{2n} \|Ax - b\|^2 + \lambda R_a(x)$. By utilizing the convex conjugate of the quadratic loss function, \cite{du2019linear} transforms this into a minimax structure
\begin{equation} \label{eq:du_formulation}
    \min_{x \in \mathbb{R}^d} \max_{y \in \mathbb{R}^n} \quad \lambda R_a(x) + \frac{1}{n} y^\top A x - \left( \frac{1}{2n}\|y\|^2 + \frac{1}{n}b^\top y \right).
\end{equation}
Here, the primal function $f(x) = \lambda R_a(x)$ is smooth and convex. We extend the formulation of \cite{du2019linear} by incorporating joint linear constraints, as studied in \cite{dai2024optimality}. The resulting linearly constrained minimax problem is
\begin{equation} \label{eq:constrained_minimax}
\begin{aligned}
    \min_{x \in \mathbb{R}^d} \max_{y \in \mathbb{R}^n} \quad & \lambda R_a(x) + \frac{1}{n} y^\top A x - \left( \frac{1}{2n}\|y\|^2 + \frac{1}{n}b^\top y \right) \\
    \textrm{subject to} \quad & \mathcal{C}x + \mathcal{D}y + e = 0,
\end{aligned}
\end{equation}
where $\mathcal{C}$ and $\mathcal{D}$ are constraint matrices.

\subsection{Related work}
Numerous algorithmic approaches have been developed for unconstrained stochastic minimax problems, spanning both zero-order methods \cite{akhavan2021distributed,dvinskikh2022gradient,xu2021zeroth,xu2023unified} and first-order strategies \cite{huang2022new,salmeron2004analysis,luo2021near,xian2021faster}. From a theoretical perspective, ~\cite{shapiro2002minimax} established an equivalence between minimax stochastic programs and expected value problems under worst-case distributions, thereby bridging robust optimization with Bayesian methodologies. \cite{nemirovski2009robust} studied a stochastic primal-dual method, which achieves a convergence rate of $O(1/\sqrt{T})$ for the duality gap in general convex-concave settings. The convergence rate can be improved by structural properties such as smoothness in certain component functions \cite{juditsky2011solving,zhao2022accelerated} or a bilinear objective structure \cite{chen2014optimal,dang2014randomized}. For instance, \cite{zhao2022accelerated} investigates a class of objectives of the form $f(x) + g(x) + \phi(x,y) - J(y)$, where smoothness is assumed for $f$ and $\phi$, and strong convexity is imposed on $f$ when applicable, leading to optimal or near-optimal complexity bounds for a stochastic primal-dual hybrid algorithm. Our work extends this research direction by considering constrained stochastic min-max problems with strongly convex–strongly concave composite objectives.
In the context of constrained optimization, many studies focus on the case with separable constraints on $x$ and $y$ \cite{chambolle2011first,arrow1958studies,chambolle2016ergodic,lin2020near,chambolle2016introduction,nemirovski2004prox}. \cite{madsen1978linearly} pioneered an early and influential feasible-point method for the linearly constrained nonlinear minimax problem. A notable advance was made by \cite{dai2024optimality}, who proposed the PGmsAD method, attaining the iteration complexity of $\mathcal{O}(\epsilon^{-2} \log \epsilon^{-1})$ for nonsmooth problems with linearly coupled constraints of the form $Ax + By + c = 0$. However, their approach is largely confined to linear coupling terms and does not extend to more general forms of variable interaction. Theoretically, Tsaknakis et al.~\cite{tsaknakis2023minimax} provided a rigorous duality analysis for linearly constrained minimax problems, demonstrating the NP-hardness and the failure of the classical max-min inequality even under strong-convexity–strong-concavity. Moreover, deterministic methods often encounter computational limitations in large-scale settings, motivating the development of stochastic alternatives.
Noted that deterministic proximal point frameworks are widely used for solving nonsmooth optimization problems, we introduce a stochastic proximal point algorithm framework for coupling linear constrained minimax optimization. For solving the nonsmooth nonlinear equations in the subproblems, the semismooth Newton  method is used inspired by \cite{milzarek2024semismooth}.

\subsection{Contributions}
Our main contributions addressed in this article are as follows.

Firstly, inspired by the implicit updates widely used in noisy or ill-conditioned settings, we develop a implicit stochastic proximal point method for the nonsmooth minimax problem \eqref{problem1}. Stochastic sampling, including both with-replacement and without-replacement schemes, constructs unbiased first-order oracle estimates of the gradient, which are then integrated with SVRG-style variance reduction in stochastic proximal point algorithmic framework. In contrast to prior works \cite{dai2024optimality} and \cite{tsaknakis2023minimax} that predominantly focus on deterministic settings, our implicit stochastic  method offers superior stability and robustness.

Secondly, to accelerate the solution of subproblems arising from the implicit updates, we study a semismooth Newton method with Armijo line search, which achieves fast local convergence in the inner iterations. We establish a worst-case computational complexity bound of $\mathcal{O}(\log(1/\varepsilon))$ Newton steps in subproblem to reach an $\varepsilon$-accurate solution. 

Thirdly, we prove global linear convergence in the strongly convex case. Under strong convexity–strong convexity–concavity assumptions, we prove that the variables $(x, y)$ achieve global q-linear convergence, while the dual variable $\lambda$ associated with the linear constraints attains r-linear convergence, provided with constant step sizes and the subproblems solved to sufficient accuracy.

Finally, the stochastic proximal point method demonstrates broad applications across constrained linear regression, network interdiction, and so on. By establishing the equivalence between the coupling linear constrained minimax optimization and the unconstrained minimax optimization, the algorithm serves as a general solver that requires no problem-specific structures. Numerical experiments confirm the method more favorable performance and enhanced robustness in practical cases.

\subsection{Orgnization}
The paper is organized as follows. Section \ref{sec:2} formulates the problem and proposes the stochastic semismooth Newton proximal point for minimax (SNmMSPP)  framework with variance reduction. Section \ref{Sloving_subproblem} develops a deterministic semismooth Newton method to efficiently solve the subproblems. Section \ref{sec:controllinexact} addresses error control and inner loop complexity. Section \ref{sec:convergence_analysis} establishes global $q$-linear convergence for primal iterates and $r$-linear convergence for multipliers. Finally, Section \ref{sec:num_experiments} demonstrates the algorithm's efficiency on linear regression and adversarial network flow problems.

\subsection{Preliminaries}
For $ N \in \mathbb{N} $, we set $ [N] := \{1, \ldots, N\} $ and denote by $ I \in \R^{n \times n} $ the identity matrix. By $ \langle \cdot, \cdot \rangle $ and $ \| \cdot \| $, we denote the standard Euclidean inner product and norm.
Let $B_\epsilon(x)$ denote the closed ball of radius $\epsilon > 0$ centered at $x$. For any $x \in \R^{n}$, $\nabla h(x)$ and $\partial g(x)$ represent the gradient of the smooth function $h$ at $x$ and the subderivative of the nonsmooth function $g$ at $x$, respectively. The conjugate function of $f : \R^{n} \rightarrow \R^{m}$ is defined as $f^{*} : \R^{m} \rightarrow \R^{n}$. $\operatorname{dom} h$ denotes the domain of the proper lower semicontinuous function $h$. $\operatorname{dist}(x, C)$ represents the distance from the point $x \in \R^{n}$ to the set $C \subseteq \R^{n}$.
The function $ f : \R^{n} \to (-\infty, \infty] $ is called \emph{$ \rho $-weakly convex}, $ \rho > 0 $, if the mapping $ x \mapsto f(x) + \frac{\rho}{2} \| x \|^{2} $ is convex. Furthermore, $ f $ is called \emph{$ \mu $-strongly convex}, $ \mu > 0 $, if $ f - \frac{\mu}{2} \| \cdot \|^{2} $ is a convex function. The set $ \operatorname{dom}(f) := \{ x \in \R^{n} : f(x) < +\infty \} $ denotes the effective domain of $ f $.

Given a function \( g:\R^{n}\rightarrow(-\infty,\infty] \), the \textbf{proximal mapping} of g is the operator given by
\[ 
\prox_{g}(x):=\underset{z\in\R^{n}}{\arg\min}\, \left\{g(z)+\frac{1}{2}\|x-z\|^{2}\right\}.
\]
given a peoper closed convex function $f:\R^n \to (-\infty,+\infty]$ and $\lambda > 0,$ the \textbf{Moreau envelope} of f is the 
function
\begin{equation*}
\env_{g}^{\lambda}:\R^{n}\rightarrow\R,\quad \env_{g}^{\lambda}(x):=\min_{z\in\R^{n}}
\left\{g(z)+\frac{1}{2\lambda}\|x-z\|^{2}\right\}. 
\end{equation*}
In particular, \( \prox_{g} \) is Lipschitz continuous with constant 1. Moreover, by \cite[Thm. 6.60]{beck2017first} the Moreau envelope is continuously differentiable with
\begin{equation*}
\nabla\env_{g}(x)=x-\prox_{g}(x). 
\end{equation*}
For a function \( g:\R^{n}\rightarrow(-\infty,\infty] \), the conjugate of \( g \) is defined by \( g^{*}:\R^{n}\rightarrow(-\infty,\infty] \), \( g^{*}(x):=\sup_{z\in\R^{n}}\,\langle z,x\rangle-g(z) \).
\begin{prop} \cite[prop~ 2.2]{milzarek2024semismooth},
Let \( g:\R^{n}\rightarrow(-\infty,\infty] \) be proper and closed. If \( g \) is \( \mu \)-strongly convex, then its conjugate \( g^{*} \) is closed, convex, proper, and Fréchet differentiable and its gradient is given by \( \nabla g^{*}(x)=\arg\max_{z\in\mathbb{R}^{n}}\langle z,x\rangle-g(z) \). In addition, \( \nabla g^{*}:\R^{n}\rightarrow \R^{n} \) is Lipschitz continuous with Lipschitz constant \( \mu^{-1} \).
\end{prop}

\begin{defn}
Let $F : V \to \R^m$ be locally Lipschitz and let $V \subseteq \R^n$ be an open set. $F$ is called semismooth at $x \in V$ (with respect to $\partial F$), if $F$ is directionally differentiable at $x$ and if it holds that
\[
\sup_{M \in \partial F(x+s)} \| F(x+s) - F(x) - Ms \| = o(\|s\|) \quad \text{as } s \to 0.
\]
Moreover, for $\nu > 0$, $F$ is called $\nu$-order semismooth (strongly semismooth if $\nu = 1$) at $x \in V$ (w.r.t. $\partial F$), if $F$ is directionally differentiable at $x$ and we have
\[
\sup_{M \in \partial F(x+s)} \| F(x+s) - F(x) - Ms \| = \mathcal{O}(\|s\|^{1+\nu}) \quad \text{as } s \to 0.
\]
\end{defn}

For problem \eqref{problem1}, we introduce the proximal gradient mapping as a measure of stationarity, i.e., for $\alpha > 0$, we define
\[
F_{\text{nat}}^\alpha : \R^n \to \R^n, \quad F_{\text{nat}}^\alpha (x) := x - \operatorname{prox}_{\alpha \varphi} (x - \alpha \nabla f(x))
\]
and $F_{\text{nat}}(x) := F_{\text{nat}}^1(x)$. Note that $x^*$ is a stationary point if and only if $F_{\text{nat}}^\alpha(x^*) = 0$. Clearly, if $f$ is $L$-smooth, then the function $F_{\text{nat}}$ is Lipschitz continuous with constant $2 + L$.

\section{The Stochastic Proximal Point Method of Minimax Optimization Problem.}\label{sec:2}
\subsection{Assumption}
We first specify the basic assumptions under which we construct and study our stochastic proximal point method. Throughout this paper, we assume that the functions $g_i: \R^n \to \R$,$h_i: \R^m \to \R$ and $f_i:\R^n \times \R^m \to \R$, $i \in [N]$, are continuously differentiable and $\varphi:\Re^n \to \Re,\psi:\Re^m \to \Re$ is a closed, convex, and proper mapping. Further conditions on $f,g,h,\varphi,\psi$ are summarized and stated below.
\begin{assump}    
\label{eq:assume_2.1}   
Let functions $g_i: \R^n \to \R$,$h_i: \R^m \to \R$ and $f_i:\R^n \times \R^m \to \R $ satisfy the following conditions.    
\begin{enumerate}[label=(A\arabic*)]        
\item Functions $g_i$,$h_i$,$f_i$ are continuously differentiable convex with Lipschitz continuous gradients;
~i.e., there exist constants $L_g^i > 0$ , $L_h^i > 0$ and $L_f^i >0$ such that for any $(x',y'), (x,y) \in \R^n \times \R^m$        
\begin{align*}            
\|\nabla g_i(x') - \nabla g_i(x)\| &\leq L_g^i \|x' - x\| \\            
\|\nabla h_i(y') - \nabla h_i(y)\| &\leq L_h^i \|y' - y\|  \\            
\max\{\|\nabla_x f_i(x',y') - \nabla_x f_i(x,y)\|,\|\nabla_y f_i(x',y')&-\nabla_y f_i(x,y)\|\} \leq L_f^i \|(x',y') - (x,y)\|          
\end{align*}        
For each $i \in [N]:=\{1,2,\cdots,N\}$, given any fixed $y \in \R^m$, the function $g_i(x) + f_i(x,y)$ is $\mu_x^i$-strongly convex w.r.t $x$;        
For each $i$, given any fixed $x \in \R^n$, the function $-h_i(y) + f_i(x,y)$ is $\mu_y^i$-strongly concave w.r.t $y$.        
Function \( f_i \) is convex with respect to \( x \) and concave with respect to \( y \).        
\label{asumme:L-smooth}        
\item Define Lagrange Functions 
\begin{equation}\label{eq:lagrange}
L(x,y,\lambda) := \varphi(x) + g(x) + f(x,y) - h(y) - \psi(y) + \langle \lambda, Ax + By + c \rangle
\end{equation}         
is closed ,where $\lambda \in \Re^q$ is a Lagrange multiplier,        
admits a saddle point $(x^*, y^*,\lambda^*)$, i.e., there exists         
$(x^*, y^*, \lambda^*) \in \Re^n \times \Re^m \times \Re$ such that        
\begin{equation*}         
\forall x\in\Re^n, y\in\mathbb \Re^m, \lambda\in\mathbb \Re^q,~          
L(x^*, y, \lambda^*) \leq L(x^*, y^*, \lambda^*) \leq L(x, y^*, \lambda)          
\label{(A2)}         
\end{equation*}    
\end{enumerate}
\end{assump}
\quad Under Assumptions (A1) and (A2), obviously, for $L(x,y,\lambda)$ being strongly convex w.r.t x and strongly concave w.r.t y
, we give the following definition
\[\textbf{$\mu_x := \max_i \mu^i_x$},\quad \textbf{$\mu_y := \max_i \mu^i_y$}.\] 
Then $L(x,y,\lambda)$ is $\mu_x$-strongly convex w.r.t $x$  
and $\mu_y$-strongly concave w.r.t $y$.  
Obviously,the mapping $x \mapsto \operatorname{prox}_{\alpha\varphi}(x)$ is semismooth for all $\alpha > 0$ and all $x \in \Re^n$,
The mapping $y \mapsto \operatorname{prox}_{\alpha\psi}(y)$ is semismooth for all $\alpha > 0$ and all $y \in \Re^m$.
Prior to the following assumptions, for simplicity of expression, we provide the following definitions. Regarding the smooth functions involving \(x\) and \(y\) in the objective function, we denote them as follows
\(
\phi_{i}^x : (x,y) \mapsto g_i(x) + f_i(x, y),~~\phi_y^x=\frac{1}{N}\sum_{i=1}^{N}\phi_{i}^x,
\phi_{i}^y : (x,y) \mapsto h_i(y) - f_i(x, y),~~\phi^y_x=\frac{1}{N}\sum_{i=1}^{N}\phi_{i}^y.
\)
For notational convenience, we denote $\phi^x_i(x,y)$ as $\phi^x_{y,i}(x)$ and $\phi^y_i(x,y)$ as $\phi^y_{x,i}(y)$; for example, $\phi^x_i(x^k,y^k)$ is denoted as $\phi^x_{y^k,i}(x^k)$.

We work with the following assumptions for the conjugates
\begin{assump}
\label{eq:assume_2.2}       
Let functions $g_i: \R^n \to \R$,~$h_i: \R^m \to \R$ and $f_i:\R^n \times \R^m \to \R $ satisfy the following conditions.        
Note that in the definitions of the conjugates, $(\phi^x_{y,i})^*$ is the conjugate of $\phi^x_{y,i}$ with respect to $x$ (with $y$ fixed), and $(\phi^y_{x,i})^*$ is the conjugate of $\phi^y_{x,i}$ with respect to $y$ (with $x$ fixed).
\begin{enumerate}[label=(A\arabic*),start=3]            
\item The functions $(\phi^x_{y,i})^*$,$(\phi^y_{x,i})^*$ are essentially differentiable \cite{goebel2008local}  with locally Lipschitz continuous gradients on the sets  
$\mathcal{D}^x_i := \text{int(dom)}(\phi^x_{y,i})^* \neq \emptyset,\mathcal{D}^y_i := \text{int(dom)}(\phi^y_{x,i})^* \neq \emptyset, i \in [N]$.            
\label{essentially_differentiable}             
\item The mappings $\nabla (\phi^x_{y,i})^* $,$\nabla (\phi_{x,i}^y)^*$ are semismooth on $\mathcal{D}^x_i$,$\mathcal{D}^y_i$ for all $i$.
\label{asumme:conjugate_semidef}            
\end{enumerate}    
\end{assump}

  Assumption \ref{asumme:L-smooth} ensures that $\phi_{y,i}^x$ and $\phi_{x,i}^y$ are Lipschitz smooth, which implies their conjugates are strongly convex. Combined with Assumption \ref{essentially_differentiable}, this guarantees that the generalized Hessians of the conjugates are uniformly positive definite. Specifically, there exist constants $\mu^*_x, \mu^*_y > 0$ (bounded by the inverse smoothness constants) such that for all $i \in [N]$
\begin{equation}
\label{eq:positive_definite}
\begin{aligned}    
\langle h, M h \rangle &\geq \mu^*_x\|h\|^2, \quad \forall M \in \partial(\nabla (\phi_{y,i}^x)^*)(u), \quad \forall u \in \mathcal{D}^x_i,\\    
\langle h, M h \rangle &\geq \mu^*_y\|h\|^2, \quad \forall M \in \partial(\nabla (\phi_{x,i}^y)^*)(v), \quad \forall v \in \mathcal{D}^y_i.
\end{aligned}
\end{equation}

\subsection{Implicit Proximal Point Method}
\label{sec:algorithm}
\qquad Inspired by traditional deterministic implicit proximal point methods \cite{eckstein1992douglas}, which offer enhanced stability, we extend this approach to develop a  implicit proximal point methods for solving minimax problem \eqref{problem1} with nonsmooth terms.

Given a suitable step size $\alpha_k > 0$, we observe from \ref{asumme:L-smooth} that the function $x \mapsto L(x,y,\lambda)$ is convex, while $y \mapsto L(x,y,\lambda)$ is concave.
Therefore, the iteration
\begin{equation*}
    \begin{split}
    x^{k+1} = \underset{x \in \Re^n}{\operatorname{argmin}} \Biggl\{ 
        &g(x)  + f(x, y^{k+1}) + (\lambda^k)^T A x+ \varphi(x) + \frac{1}{2\alpha_k}\|x - x^k\|^2  \Biggr\},
    \end{split}
    \end{equation*}
 The first-order optimality condition for $x^{k+1}$ is given by
\begin{align*}
    0 &\in \nabla g(x^{k+1}) + \nabla_x f(x^{k+1}, y^{k+1}) + A^\top \lambda^k + \partial \varphi(x^{k+1}) + \frac{1}{\alpha_k}(x^{k+1} - x^k) \\
    \iff &x^{k+1} = \operatorname{prox}_{\alpha_k \varphi} \big( x^k - \alpha_k [ \nabla g(x^{k+1}) + \nabla_x f(x^{k+1}, y^{k+1}) + A^\top \lambda^k ] \big).
\end{align*}
Similarly, 
    \(
        y^{k+1}=\operatorname{prox}_{\alpha_k \psi}\left(y^k+\alpha_k[\nabla_y f(x^k,y^{k+1})-\nabla h(y^{k+1})+B^T\lambda^k]\right)
    \)
    is well-defined.
    
    We reformulate the above update steps into the implicit update system of equations
    \begin{equation}
        \begin{cases}
            \displaystyle
            y^{k+1} = \operatorname{prox}_{\alpha_k \psi} \Big( y^k + \alpha_k \big[ \nabla_y f(x^k, y^{k+1}) - \nabla h(y^{k+1}) + B^\top \lambda^k \big] \Big), \\
            \displaystyle
            x^{k+1} = \operatorname{prox}_{\alpha_k \varphi} \Big( x^k - \alpha_k \big[ \nabla g(x^{k+1}) + \nabla_x f(x^{k+1}, y^{k+1}) + A^\top \lambda^k \big] \Big), \\
            \displaystyle
            \lambda^{k+1} = \lambda^{k} - \alpha_k \big( A x^{k+1} + B y^{k+1} + c \big).
        \end{cases}\tag{2.4} \label{eq:iteration_form}
    \end{equation}
    The update for $\lambda$ employs gradient descent to handle the linear constraint $Ax + By + c = 0$ in problem \eqref{problem1}, where the term $(A x^{k+1} + B y^{k+1} + c)$ represents the constraint violation at the current primal iterates.

We now introduce an equation-based characterization of the implicit update \eqref{eq:iteration_form}.We now define   
\( \bar{\xi}_{y}^{x,k+1} = (\bar{\xi}_{y,1}^{x,k+1}, \ldots, \bar{\xi}_{y,N}^{x,k+1}) \) and    
\( \bar{\xi}_{x}^{y,k+1} = (\bar{\xi}_{x,1}^{y,k+1}, \ldots, \bar{\xi}_{x,N}^{y,k+1}) \) by
\[
\begin{cases}
\bar{\xi}_{y,i}^{x,k+1} := \nabla_x \phi^x_{y^{k+1},i}(x^{k+1}) = \nabla g_{i}(x^{k+1})+ \nabla_x f_{i}(x^{k+1},y^{k+1}),\\
\bar{\xi}_{x,i}^{y,k+1} := \nabla_y \phi^y_{x^k,i}(y^{k+1}) = \nabla h_{i}(y^{k+1})- \nabla_y f_{i}(x^{k},y^{k+1}),
\tag{2.5}
\quad i \in N.
\end{cases}
\]
Under assumption \ref{asumme:L-smooth}, \cite[Thm. 4.20]{beck2017first} yields
\[
\bar{\xi}_{y,i}^{x,k+1}= \nabla_x \phi^x_{y,i}(x^{k+1}) \quad \iff \quad \nabla (\phi^x_{y,i})^*(\bar{\xi}_{y,i}^{x,k+1})=x^{k+1},
\]
\[
\bar{\xi}_{x,i}^{y,k+1}= \nabla_y \phi^y_{x,i}(y^{k+1}) \quad \iff \quad \nabla (\phi^y_{x,i})^*(\bar{\xi}_{x,i}^{y,k+1})=y^{k+1}.
\]
The system \eqref{eq:iteration_form} is equivalent to the system
\[
\begin{cases}
\label{equation_11}
y^{k+1} = \operatorname{prox}_{\alpha_k \psi} \left( y^k- \frac{\alpha_k}{N}\sum_{i=1}^{N}\bar{\xi}_{x,i}^{y,k+1}+\alpha_k B^\top \lambda^k  \right),\\
x^{k+1} = \operatorname{prox}_{\alpha_k \varphi} \left( x^k - \frac{\alpha_k}{N}\sum_{i=1}^{N}\bar{\xi}_{y,i}^{x,k+1}-\alpha_k A^\top \lambda^k \right), \\
\lambda^{k+1}=\lambda^k-\alpha_k(Ax^{k+1}+By^{k+1}+c),\\
\nabla (\phi^y_{x,i})^*(\bar{\xi}_{x,i}^{y,k+1})=\operatorname{prox}_{\alpha_k \psi} \left( y^k -\frac{\alpha_k}{N}\sum_{i=1}^{N}\bar{\xi}_{x,i}^{y,k+1}+\alpha_k B^\top \lambda^k\right),\\
\nabla(\phi^x_{y,i})^*(\bar{\xi}_{y,i}^{x,k+1})=\operatorname{prox}_{\alpha_k \varphi} \left( x^k - \frac{\alpha_k}{N}\sum_{i=1}^{N}\bar{\xi}_{y,i}^{x,k+1}-\alpha_k A^\top \lambda^k \right),  i\in N.
\tag{2.6}
\end{cases}
\]

The update for $\lambda$ employs gradient descent to handle the linear constraint $Ax+By+c=0$ in problem~\eqref{problem1}, where the term $Ax^{k+1}+By^{k+1}+c$ represents the constraint violation at the current primal iterates.
\subsection{Stochastic Gradient Approximation And Variance Reduction}

To address the high computational cost of full gradients in Problem~\eqref{problem1}, we employ stochastic approximations. For a random subset $\mathcal{S} \subseteq [N]$, we define the stochastic estimators for the functions and their gradients as 
$g_{\mathcal{S}}(x) := \tfrac{1}{|\mathcal{S}|}\sum_{i\in\mathcal{S}}g_i(x)$, $\nabla g_{\mathcal{S}}(x) := \tfrac{1}{|\mathcal{S}|}\sum_{i\in\mathcal{S}}\nabla g_i(x)$, $h_{\mathcal{S}}(y) := \tfrac{1}{|\mathcal{S}|}\sum_{i\in\mathcal{S}}h_i(y)$, $\nabla h_{\mathcal{S}}(y) := \tfrac{1}{|\mathcal{S}|}\sum_{i\in\mathcal{S}}\nabla h_i(y)$, $\nabla_x f_{\mathcal{S}}(x,y) := \tfrac{1}{|\mathcal{S}|}\sum_{i\in\mathcal{S}}\nabla_x f_i(x,y)$, $\nabla_y f_{\mathcal{S}}(x,y) := \tfrac{1}{|\mathcal{S}|}\sum_{i\in\mathcal{S}}\nabla_y f_i(x,y)$.
Accordingly, the stochastic Lagrangian is defined as
\[
L_\mathcal{S}(x,y,\lambda) := \varphi(x) + g_{\mathcal{S}}(x) + f_{\mathcal{S}}(x,y) - h_{\mathcal{S}}(y) - \psi(y) + \langle \lambda, A x + B y + c \rangle.
\]
where $ f_{\mathcal{S}}(x,y) := \frac{1}{|\mathcal{S}|}\sum_{i\in\mathcal{S}}f_i(x,y).$

Replacing the full gradients in \eqref{eq:iteration_form} with stochastic estimators over a random batch $\mathcal{S}_k \subseteq [N]$ ($\mathcal{S}_k$ be a randomly sampled subset at iteration $k$) yields the update
\begin{equation}
\label{sto_imp_update} \tag{2.7}
\begin{cases}
    y^{k+1} = \operatorname{prox}_{\alpha_k \psi} \big( y^k + \alpha_k [ \nabla_y f_{\mathcal{S}_k}(x^k, y^{k+1}) - \nabla h_{\mathcal{S}_k}(y^{k+1}) + B^\top \lambda^k ] \big), \\
    x^{k+1} = \operatorname{prox}_{\alpha_k \varphi} \big( x^k - \alpha_k [ \nabla g_{\mathcal{S}_k}(x^{k+1}) + \nabla_x f_{\mathcal{S}_k}(x^{k+1}, y^{k+1}) + A^\top \lambda^k ] \big), \\
    \lambda^{k+1} = \lambda^k - \alpha_k ( A x^{k+1} + B y^{k+1} + c ).
\end{cases}
\end{equation}
Next, we define the natural residual maps $F_{y,\text{nat}}^{x,\alpha}: \R^n \to \R^n$ and $F_{x,\text{nat}}^{y,\alpha}: \R^m \to \R^m$ as
\begin{align*}
F_{y,\text{nat}}^{x,\alpha}(x) &:= x - \operatorname{prox}_{\alpha \varphi} \big( x - \alpha ( \nabla g(x) + \nabla_x f(x,y) + A^\top \lambda ) \big), \\
F_{x,\text{nat}}^{y,\alpha}(y) &:= y - \operatorname{prox}_{\alpha \psi} \big( y + \alpha ( \nabla_y f(x,y) - \nabla h(y) + B^\top \lambda ) \big).
\end{align*}
Let $F^{x}_{y,\text{nat}} := F_{y,\text{nat}}^{x,1}$ and $F^{y}_{x,\text{nat}} := F_{x,\text{nat}}^{y,1}$. Under the smoothness assumptions on $f, g, h$, the maps $F^{x}_{y,\text{nat}}$ and $F^{y}_{x,\text{nat}}$ are Lipschitz continuous with constants $2 + L_f + L_g$ and $2 + L_f + L_h$, respectively.

While the basic stochastic approach reduces computational burden, the inherent variance in gradient estimation can lead to slower convergence and instability. To address this limitation and achieve faster convergence rates, we incorporate variance reduction techniques that have proven to be powerful tools for accelerating stochastic optimization algorithms \cite{j2016proximal,xiao2014proximal}. Specifically, we consider SVRG-type stochastic oracles that additionally incorporate the following gradient correction terms in each iteration
\begin{equation*}
\begin{aligned}
       v^{x,k}_y &:=\nabla g(\tilde{x})+\nabla_x f(\tilde{x},\tilde{y})-\nabla g_{\mathcal{S}_k}(\tilde{x})-\nabla_x f_{\mathcal{S}_k}(\tilde{x},\tilde{y}) \\
v^{y,k}_x &:=-\nabla h(\tilde{y})+\nabla_y f(\tilde{x},\tilde{y})+\nabla h_{\mathcal{S}_k}(\tilde{y})-\nabla_y f_{\mathcal{S}_k}(\tilde{x},\tilde{y}),
\label{svrg_varience_reduce_term}
\end{aligned}
\tag{2.8}
\end{equation*}
where $\tilde{x},\tilde{y}$ is a reference point that is generated in an outer loop. To simplify the notation,
setting 
\[\hat{v}^{x,k}_y := \alpha_k (v^{x,k}_y+A^T \lambda^k),\quad 
\hat{v}^{y,k}_x := -\alpha_k (v^{y,k}_x+B^T \lambda^k)\tag{2.9}\label{hat{v}}.\]
 This leads to stochastic proximal point-type updates of the form

\begin{equation}
    \tag{2.10}
    \begin{cases}
        \label{(3.3)}
        \displaystyle
        y^{k+1} = \operatorname{prox}_{\alpha_k \psi} \Big( y^k - \alpha_k \big[ -\nabla_y f_{\mathcal{S}_k}(x^k, y^{k+1}) + \nabla h_{\mathcal{S}_k}(y^{k+1})\big]-\hat{v}^{y,k}_x  \Big), \\
        \displaystyle
        x^{k+1} = \operatorname{prox}_{\alpha_k \varphi} \Big( x^k - \alpha_k \big[ \nabla g_{\mathcal{S}_k}(x^{k+1}) + \nabla_x f_{\mathcal{S}_k}(x^{k+1}, y^{k+1})\big] -\hat{v}^{x,k}_y  \Big), \\
        \displaystyle
        \lambda^{k+1} = \lambda^k - \alpha_k \big( A x^{k+1} + B y^{k+1} + c \big).
    \end{cases}
\end{equation}

\indent We now introduce an alternative equation-based characterization of the implicit update \eqref{sto_imp_update}.
Let us set \( b_k := |\mathcal{S}_k| \) and let \( (\kappa_k(1), \ldots, \kappa_k(b_k)) \)
 enumerate the elements of the tuple \( \mathcal{S}_k \). We will often abbreviate \( \kappa_k \)
  by \( \kappa \). We now define 
  \( \xi_{y}^{x,k+1} = (\xi_{y,1}^{x,k+1}, \ldots, \xi_{y,b_k}^{x,k+1}) \) and 
   \( \xi_{x}^{y,k+1} = (\xi_{x,1}^{y,k+1}, \ldots, \xi_{x,b_k}^{y,k+1}) \) by
\[
\begin{cases}
\xi_{y,i}^{x,k+1} := \nabla_x \phi^x_{y^{k+1},k(i)}(x^{k+1}) = \nabla g_{k(i)}(x^{k+1})+ \nabla_x f_{k(i)}(x^{k+1},y^{k+1})\\
\xi_{x,i}^{y,k+1} := \nabla_y \phi^y_{x^k,k(i)}(y^{k+1}) = \nabla h_{k(i)}(y^{k+1})- \nabla_y f_{k(i)}(x^{k},y^{k+1}),
\tag{2.11}
\quad i \in [b_k].
\end{cases}
\]
Under assumption \ref{asumme:L-smooth}, \cite[Thm 4.20]{beck2017first} yields
\[
\xi_{y,i}^{x,k+1}= \nabla_x \phi^x_{y,k(i)}(x^{k+1}) \quad \iff \quad \nabla (\phi^x_{y,k(i)})^*(\xi_{y,i}^{x,k+1})=x^{k+1},
\]
\[
\xi_{x,i}^{y,k+1}= \nabla_y \phi^y_{x,k(i)}(y^{k+1}) \quad \iff \quad \nabla (\phi^y_{x,k(i)})^*(\xi_{x,i}^{y,k+1})=y^{k+1}.
\]

The step \eqref{(3.3)} is equivalent to the system

\[
\begin{cases}
\label{equation_1}
y^{k+1} = \operatorname{prox}_{\alpha_k \psi} \left( y^k- \frac{\alpha_k}{b_k}\sum_{i=1}^{b_k}\xi_{x,i}^{y,k+1}-\hat{v}^{y,k}_x \right),\\
x^{k+1} = \operatorname{prox}_{\alpha_k \varphi} \left( x^k - \frac{\alpha_k}{b_k}\sum_{i=1}^{b_k}\xi_{y,i}^{x,k+1}-\hat{v}^{x,k}_y \right), \\

\lambda^{k+1}=\lambda^k-\alpha_k(Ax^{k+1}+By^{k+1}+c),\\
\nabla (\phi^y_{x,k(i)})^*(\xi_{x,i}^{y,k+1})=\operatorname{prox}_{\alpha_k \psi} \left( y^k -\frac{\alpha_k}{b_k}\sum_{i=1}^{b_k}\xi_{x,i}^{y,k+1}-\hat{v}^{y,k}_x\right),\\
\nabla(\phi^x_{y,k(i)})^*(\xi_{y,i}^{x,k+1})=\operatorname{prox}_{\alpha_k \varphi} \left( x^k - \frac{\alpha_k}{b_k}\sum_{i=1}^{b_k}\xi_{y,i}^{x,k+1}-\hat{v}^{x,k}_y \right),  i\in [b_k].
\tag{2.12}
\end{cases}
\]
 The semismooth Newton method for \eqref{equation_1} is
specified and discussed in the next section. In this article, we primarily focus on the variance-reduced
update (\ref{(3.3)}), yet the technique and results presented in \hyperref[Sloving_subproblem]{Section\ref{Sloving_subproblem}} also hold true for the general
update \eqref{sto_imp_update}.
\subsection{Bound of Variance}
We derives the variance bounds for our stochastic gradient estimators, a critical step for establishing convergence guarantees. We first analyze the variance properties under different sampling schemes, then combine these results with Lipschitz continuity to obtain the final bounds. The analysis ensures our gradient estimators maintain controlled variance while preserving unbiasedness, providing the theoretical foundation for algorithm stability.
We now formally specify the notion of admissible stochastic oracles for our problem.

\begin{defn}
Let \( \mathcal{S} \sim \mathbb{P} \) be a \( b \)-tuple of elements of
 \( [N] \), where \( b \in [N] \) is fixed. Let \( \kappa(i) \in [N] \)
  denote the random number in the \( i \)-th position of \( \mathcal{S} \).
   We call \( \mathcal{S} \sim \mathbb{P} \) an admissible sampling procedure if, f
   or all \( z_i \in \mathbb{R}^\ell \), \( i \in [N] \), \( \ell \in \mathbb{N} \),
    it holds that \( \mathbb{E}_{\mathbb{P}}[z_{\mathcal{S}}] = \frac{1}{N} \sum_{i=1}^N z_i \) 
    where \( z_{\mathcal{S}} := \frac{1}{b} \sum_{i=1}^b z_{\kappa(i)} \).
\label{def:admissible_sampling}
\end{defn}
If \( \mathcal{S} \sim \mathbb{P} \) is an admissible sampling procedure, 
then we have \( \mathbb{E}_{\mathbb{P}}[f_{\mathcal{S}}(x)] = f(x) \) and 
\( \mathbb{E}_{\mathbb{P}}[\nabla f_{\mathcal{S}}(x)] = \nabla f(x) \) for all \( x \in \mathbb{R}^n \).
In the simplest case, we can choose \( \mathcal{S} \) by drawing \( b \) elements from \( [N] \)
 under a uniform distribution (cf. \cite{xiao2014proximal} for a similar setting). This is an admissible sampling procedure 
 in the sense of \hyperref[def:admissible_sampling]{Definition \ref{def:admissible_sampling}}, regardless of whether we draw with or without replacement  \cite[\S 2.8]{lohr2021sampling}.
\begin{rem}
Note that \( x^k \), \( \alpha_k \), etc., serve as abbreviations when the value of \( s \) is clear. The notation \( x^{s,k} \), \( \alpha_k^s \), etc., can be used to highlight the full \( (s,k) \)-dependence.
\end{rem}

Let \((\Omega, \mathcal{F}, \mathbb{P})\) be a probability space, where \(\Omega\) is the sample space, 
\(\mathcal{F}\) is a \(\sigma\)-algebra over \(\Omega\), and \(\mathbb{P}\) is a probability measure.
 Suppose \( y_i : \Omega \to \mathcal{Y} \) are \textbf{\(\mathcal{F}\)-measurable random variables} 
 for \( i = 1, 2, \ldots, N \), where \(\mathcal{Y}\) is the measurable space (e.g., \(\mathcal{Y} = \R\)
  with the Borel \(\sigma\)-algebra). 
  
\begin{Lemma}  
\cite[\S 2.8]{lohr2021sampling}
Let \(\mathcal{S}\) be a \(b\)-tuple sampled uniformly at random \emph{without replacement} from \([N] \coloneqq \{1, 2, \dots, N\}\), independent of any other random variables. Define the indicator random variables 
\[
Z_i \coloneqq 
\begin{cases} 
1, & \text{if unit } i \text{ is included in } \mathcal{S}, \\ 
0, & \text{otherwise}, 
\end{cases}
\quad \text{for } i = 1, 2, \dots, N.
\]
Then, the collection \(\{Z_1, \dots, Z_N\}\) consists of identically distributed Bernoulli random
 variables with \(\Pr[Z_i = 1] = \frac{b}{N}\) and \(\Pr[Z_i = 0] = 1 - \frac{b}{N}\).
  The sample mean \(\bar{w} \coloneqq \frac{1}{b} \sum_{i=1}^N Z_i w_i\) is an unbiased estimator of the 
  population mean \(\bar{w}_U \coloneqq \frac{1}{N} \sum_{i=1}^N w_i\), i.e., \(\mathbb{E}[\bar{w}] = \bar{w}_U\).

The unbiased sample variance be \(s^2=\frac{1}{b-1}\sum_{i \in \mathcal{S}}(w_i-\bar{w})^2\).
The population variance be \( S^2 = \frac{1}{N - 1}\sum_{i = 1}^N (w_i - \bar{w}_U)^2 \).
 Then the variance of \( \bar{w} \) is \( V(\bar{w})
      = \left(1 - \frac{b}{N}\right)\frac{S^2}{b} \).
      \label{expectation_variance}
\end{Lemma}

Before proving the lemma, we define
\begin{align*}
&\phi_{y,\mathcal{S}}^x:= \frac{1}{|\mathcal{S}|} \sum_{i=1}^{|\mathcal{S}|} \phi_{y,i}^x,~
\phi_{x,\mathcal{S}}^y:=\frac{1}{|\mathcal{S}|}\sum_{i=1}^{|\mathcal{S}|} \phi_{x,i}^y.\\
&\zeta_{y,i}^x(x,\tilde{x};y,\tilde{y}):=\nabla_x \phi_{y,i}^x(x)-\nabla_x \phi_{\tilde{y},i}^x(\tilde{x}),\ \zeta_{y,\mathcal{S}}^x(x,\tilde{x};y,\tilde{y}):=\frac{1}{\mathcal{|S|}}\sum_{i\in\mathcal{S}}\zeta_{y,i}^x(x,\tilde{x};y,\tilde{y}),\\ 
&\zeta_{x,i}^y(y,\tilde{y};x,\tilde{x}):=\nabla_y \phi_{x,i}^y(y)-\nabla_y \phi_{\tilde{x},i}^y(\tilde{y}),\zeta_{x,\mathcal{S}}^y(y,\tilde{y};x,\tilde{x}):=\frac{1}{\mathcal{|S|}}\sum_{i\in\mathcal{S}}\zeta_{x,i}^y(y,\tilde{y};x,\tilde{x})\\
&~\bar{L}_{\phi,x}:=\max_i L_{\phi,x}^i,~\bar{L}_{\phi,y}:=\max_i L_{\phi,y}^i.
\end{align*}
Then, we have:~$\phi_{y,i}^x$ is $L_{\phi,x}^i:=(L_f^i+L_g^i)$-smooth and $\mu_x^i$-strongly convex ,
 $\phi_i^y$ is $L_{\phi,y}^i:=(L_f^i+L_h^i)$-smooth and $\mu_y^i$-strongly convex.

In this section, let $\mathcal{F}$ be a $\sigma$-algebra and suppose that $x$ and $\tilde{x}$ are $\mathcal{F}$-measurable random variables in $\R^n$,
 $y$ and $\tilde{y}$ are $\mathcal{F}$-measurable random variables in $\R^m$.

\begin{Lemma}
Suppose that condition \ref{asumme:L-smooth} is satisfied and let the index \(i\) be drawn uniformly from \([N]\) 
and independently of \(\mathcal{F}\). Conditioned on \(\mathcal{F}\), we then have 
\begin{align*}
&\mathbb{E}\|\zeta_{y,i}^x(x,\tilde{x};y,\tilde{y})\|^2 \leq \bar{L}_{\phi,x}^2[\|x - \tilde{x}\|^2+\|y-\tilde{y}\|^2],\quad\\
&\mathbb{E}\|\zeta_{x,i}^y(y,\tilde{y},x,\tilde{x})\|^2 \leq \bar{L}_{\phi,y}^2[\|x - \tilde{x}\|^2+\|y-\tilde{y}\|^2]
\end{align*} 
almost surely.
\label{control_gradient}
\end{Lemma}
\begin{proof}
The first statement follows directly from Lipschitz  smoothness.  

\end{proof}

\begin{Lemma}
\label{unbiase_guarantee}
Let \(\mathcal{S}\) be a \(b\)-tuple drawn uniformly from \([N]\) independent of \(\mathcal{F}\). 
Define the estimators
\begin{align*}
u^x_y(x,\tilde{x};y,\tilde{y}) &:= \nabla_x \phi_{y,\mathcal{S}}^x(x) - \nabla_x\phi_{\tilde{y},\mathcal{S}}^x(\tilde{x}) + \nabla_x \phi_{\tilde{y}}^x(\tilde{x}), \\
u^y_x(y,\tilde{y};x,\tilde{x}) &:= \nabla_y \phi_{x,\mathcal{S}}^y(y) - \nabla_y \phi_{\tilde{x},\mathcal{S}}^y(\tilde{y}) + \nabla_y \phi_{\tilde{x}}^y(\tilde{y}).
\end{align*}
Then, \(u_y^x\) and \(u_x^y\) are unbiased estimators of \(\nabla_x \phi_y^x(x)\) and \(\nabla_y\phi_x^y(y)\). 
Furthermore, let \(\zeta_{y,i}^x := \nabla_x \phi^x_{y,i}(x)-\nabla_x \phi^x_{\tilde{y},i}(\tilde{x})\) (and similarly for \(\zeta_{x,i}^y\)). The variance is bounded by
\begin{enumerate}
    \item[(i)] With replacement: \(\mathbb{E}\|u_y^x(x,\tilde{x};y,\tilde{y})-\nabla_x \phi^x_y(x)\|^2 \leq \frac{1}{b}\mathbb{E}\|\zeta^x_{y,i}(x,\tilde{x};y,\tilde{y})\|^2\).
    \item[(ii)] Without replacement: \(\mathbb{E}\|u_y^x(x,\tilde{x};y,\tilde{y})-\nabla_x \phi^x_y(x)\|^2 \leq \frac{N-b}{b(N - 1)}\frac{1}{N}\sum_{i=1}^{N}\mathbb{E}\|\zeta_{y,i}^x(x,\tilde{x};y,\tilde{y})\|^2\).
\end{enumerate}
Symmetric bounds hold for \(u_x^y\).
\end{Lemma}

\begin{proof}
Unbiasedness follows directly from \hyperref[expectation_variance]{Lemma \ref{expectation_variance}}. 
Note that \(\zeta_{y,\mathcal{S}}^x(x,\tilde{x};y,\tilde{y}) = u_y^x(x,\tilde{x};y,\tilde{y}) - \nabla_x \phi^x_{\tilde{y}}(\tilde{x})= \nabla_x \phi^x_{y,\mathcal{S}}(x)-\nabla_x \phi^x_{\tilde{y},\mathcal{S}}(\tilde{x})
\).

For Case (i), applying the variance bound for independent sampling \cite[Lemma 7]{j2016proximal}, we immediately obtain
\begin{align*}
\mathbb{E}\|u_y^x(x,\tilde{x};y,\tilde{y})-\nabla_x \phi^x_y(x)\|^2=& \mathbb{E}\|\frac{1}{b} \sum_{i\in \mathcal{S}}\zeta^x_{y,i}(x,\tilde{x};y,\tilde{y})-\mathbb{E}\zeta^x_{y,i}(x,\tilde{x};y,\tilde{y})\|^2\\
\leq& \frac{1}{b}\mathbb{E}\|\zeta^x_{y,i}(x,\tilde{x};y,\tilde{y})\|^2.
\end{align*}
For Case (ii), using the finite population correction from \hyperref[expectation_variance]{Lemma \ref{expectation_variance}}, we have
\[
\mathbb{E}\|u_y^x(x,\tilde{x};y,\tilde{y})-\nabla_x \phi^x_y(x)\|^2 \leq \left(1 - \frac{b}{N}\right)\frac{1}{b(N - 1)}\sum_{i=1}^{N}\mathbb{E}\|\zeta^x_{y,i}(x,\tilde{x};y,\tilde{y})\|^2.
\]
The bounds for \(u_x^y\) follow symmetrically by exchanging \(x\) and \(y\).
\end{proof}

This result guarantees unbiasedness for convergence and quantifies variance reduction, highlighting the benefit of the finite population correction $(1-b/N)$ in without-replacement sampling. Combining \hyperref[control_gradient]{Lemma \ref{control_gradient}} and \hyperref[unbiase_guarantee]{Lemma \ref{unbiase_guarantee}}, we obtain the following conclusion.

\begin{thm}
\label{thm:variance_reduce_1}
    Suppose that \ref{asumme:L-smooth} hold. Let \( \mathcal{S} \) be a \( b \)-tuple drawn uniformly from \( [N] \) and independently 
of \( \mathcal{F} \). With \( u^x_y(x,\tilde{x};y,\tilde{y}),u_x^y(y,\tilde{y};x,\tilde{x}) \) as in \hyperref[unbiase_guarantee]{Lemma \ref{unbiase_guarantee}} and conditioned on \( \mathcal{F} \), it holds that

    \[
    \mathbb{E}\|u_y^x(x,\tilde{x};y,\tilde{y})- \nabla_x \phi_y^x(x)\|^2 \leq \frac{\bar{L}^2_{\phi,x} \tau}{b} [\|x - \tilde{x}\|^2+\|y-\tilde{y}\|^2],
    \]
    \[
    \mathbb{E}\|u_x^y(y,\tilde{y};x,\tilde{x})- \nabla_y \phi_x^y(y)\|^2 \leq \frac{\bar{L}^2_{\phi,y} \tau}{b} [\|y - \tilde{y}\|^2+\|x-\tilde{x}\|^2].
    \]
where \( \tau = 1 \) if \( \mathcal{S} \) is drawn with replacement and \( \tau = \frac{N-b}{N-1} \) if \( \mathcal{S} \) is drawn without replacement.
\end{thm}
\begin{rem}
\label{rem:variance_bound}
Unlike traditional stochastic optimization methods, our analysis avoids explicit bounded variance assumptions. Instead, variance control is achieved through rigorous subproblem analysis, providing a more natural theoretical foundation for the algorithm's behavior.
\end{rem}

\subsection{Algorithm Framework}

We propose an algorithmic framework ( \hyperref[alg:SNmMSPP]{Algorithm \ref{alg:SNmMSPP}})alternates between two main steps: a variance-reduced stochastic 
gradient step and the step to solve the resulting proximal  subproblem: solving the nonsmooth equations defining $\xi^{y,k+1}_x$ and $\xi^{x,k+1}_y$. The dimension of this system, as well as the dimensions of $\xi^{y,k+1}_x$ and $\xi^{x,k+1}_y$, is controlled by the batch size $b_k$, which becomes a significant advantage when $b_k \ll n$. We allow approximate solutions to system \eqref{equation_1}, leading to inexact proximal steps.  The convergence of \hyperref[alg:SNmMSPP]{Algorithm \ref{alg:SNmMSPP}} partially depends on the solution accuracy of the subproblems; however, theoretically, the convergence rate of \hyperref[alg:SNmMSPP]{Algorithm \ref{alg:SNmMSPP}} does not impose stringent requirements on the convergence speed of the subproblems. This flexibility enables us to propose an algorithmic framework for solving subproblem without being constrained to any specific algorithm.
\begin{algorithm}[h]
\small 
\caption{SNmMSPP}
\label{alg:SNmMSPP}
\begin{algorithmic}
\Require Initial points $\tilde{x}^0, \tilde{y}^0, \tilde{\lambda}^0$; parameters $m, S, \alpha_k^s, b_k^s, \varepsilon_k^s$.
\For{$s = 0, \dots, S$}    
    \State Initialize: $x^0=x^{s,0}=\tilde{x} \leftarrow \tilde{x}^s, ~y^0=y^{s,0}=\tilde{y} \leftarrow \tilde{y}^s, ~\lambda^0=\lambda^{s,0}=\tilde{\lambda} \leftarrow \tilde{\lambda}^s$.
    \State Set step sizes $\alpha_k$, batch sizes $b_k$, tolerances $\varepsilon_k$ for inner loop.
    \For{$k = 0, \dots, m-1$}        
        \State \textbf{(Sampling)} Sample $S_k$ ($|S_k|=b_k$) and compute 
        
       \State \( v^{x,k}_{y,s}:=v^{x,k}_y,~v^{y,k}_{x,s}:=v^{y,k}_x,~\hat{v}^{x,k}_y, ~\hat{v}^{y,k}_x . \) \quad ( See \eqref{svrg_varience_reduce_term} and \eqref{hat{v}})
        
        \State \textbf{ (Update $y$)} 
        \State \quad Find $\xi^{y,k+1}_x$ via \hyperref[alg:ssn_generic]{Alg.~\ref{alg:ssn_generic}} s.t. $\|\nabla \mathcal{I}^y_x(\xi^{y,j}_x)\| \leq \varepsilon_k$ with input        
          \State     $x^k,y^k,\lambda^k,\alpha_k,S_k,-\hat{v}_x^{y,k},\varepsilon_k.$
        \State \quad $y^{k+1} \leftarrow \operatorname{prox}_{\alpha_k \psi} \big( y^k- \frac{\alpha_k}{b_k}\sum_{i=1}^{b_k}\xi_{x,i}^{y,k+1}-\hat{v}^{y,k}_x \big)$.
        
        \State \textbf{(Update $x$)} 
        \State \quad Find $\xi^{x,k+1}_y$ via \hyperref[alg:ssn_generic]{Alg.~\ref{alg:ssn_generic}} s.t. $\|\nabla \mathcal{I}^x_y(\xi^{x,j}_y)\| \leq \varepsilon_k$. with input        
         \State $x^k,y^{k+1},\lambda^k,\alpha_k,S_k,-\hat{v}_y^{x,k},\varepsilon_k.$ 
        \State \quad $x^{k+1} \leftarrow \operatorname{prox}_{\alpha_k \varphi} \big(x^k - \frac{\alpha_k}{b_k}\sum_{i=1}^{b_k}\xi_{y,i}^{x,k+1}-\hat{v}^{x,k}_y \big)$.
        
        \State \textbf{(Update $\lambda$)} 
        \State \quad $\lambda^{k+1} \leftarrow \lambda^k - \alpha_k(Ax^{k+1} + By^{k+1} + c)$.
    \EndFor    
    \State Update outer: $\tilde{x}^{s+1} \leftarrow x^m, ~\tilde{y}^{s+1} \leftarrow y^m, ~\tilde{\lambda}^{s+1} \leftarrow \lambda^m$.
\EndFor
\State \textbf{Return} $\tilde{x}^{s+1}, \tilde{y}^{s+1}, \tilde{\lambda}^{s+1}$.
\end{algorithmic}
\end{algorithm}

This potential inexactness constitutes an important component of our algorithmic design and convergence analysis, and we will address such inexact updates in \hyperref[sec:controllinexact]{Section \ref{sec:controllinexact}}.
\begin{rem}
\label{rem:constraint_satisfaction}
It is important to note that the stochastic strategy of our updates does not guarantee that the linear constraints $Ax + By + c = 0$ are satisfied at every iteration. However,  We will discuss in \hyperref[sec:convergence_analysis]{Section \ref{sec:convergence_analysis}} how to project onto the linear constraints to ensure they are satisfied within a desired tolerance throughout the algorithm.
\end{rem}

\section{A Semismooth Newton Method for Solving the Subproblem}\label{Sloving_subproblem}

In this section, we employ a deterministic semismooth Newton method to efficiently solve the subproblems arising from stochastic gradient approximations within a tolerance $\varepsilon_{\text{sub}}$. Following \cite{ulbrich2011semismooth} and \cite{qi1993nonsmooth}, we rely on the concept of surrogate generalized differentials $\hat{\partial} F$ for locally Lipschitz functions.

\begin{assump}
\label{assump:semismooth}
For a constant $\nu \in (0, 1]$, we assume the following mappings are $\nu$-order semismooth
\begin{enumerate}[label=(A\arabic*),start=5]
    \item The proximal operators $x \mapsto \mathrm{prox}_{\alpha \varphi}(x)$ on $\R^n$ and $y \mapsto \operatorname{prox}_{\alpha\psi}(y)$ on $\R^m$ for any $\alpha > 0$. \label{A5}
    \item The conjugate gradient mappings $z \mapsto \nabla (\phi_{y,i}^x)^*(z)$ on $\mathcal{D}^x_i$ and $z \mapsto \nabla (\phi_{x,i}^y)^*(z)$ on $\mathcal{D}^y_i$ for all $i$. \label{A6}
\end{enumerate}
\end{assump}

Consider a batch $\mathcal{S} = (\kappa(1), \ldots, \kappa(b)) \in [N]$,step size $\alpha > 0$, and vectors $x, v^x_y \in \R^n$, $y, v^{y,k}_x \in \R^m$, $\lambda \in \R^q$, we define the subproblem domains as $\mathcal{D}^x = \prod_{i \in \mathcal{S}} \mathcal{D}^x_i \subseteq \R^{m^x_{\mathcal{S}}}$ and $\mathcal{D}^y = \prod_{i \in \mathcal{S}} \mathcal{D}^y_i \subseteq \R^{m^y_{\mathcal{S}}}$, where $m_{\mathcal{S}}^\cdot$ denotes the cumulative dimension. 

The second part of \eqref{equation_1} is reformulated as the system 
\[
\begin{cases}
\tag{3.1}
\label{semi_smooth_equation}
\mathcal{F}^x_y(\xi^x_y)=0\\
\mathcal{F}^y_x(\xi^y_x)=0,
\end{cases}
\] where $\boldsymbol{\xi}^x = ((\xi^x_{y,1})^\top, \ldots, (\xi^x_{y,b})^\top)^\top \in \mathcal{D}^x$ (and similarly for $\boldsymbol{\xi}^y \in \mathcal{D}^y$). 
Letting $\mathcal{A}^x_{\mathcal{S}} := \frac{1}{b} [I, \ldots, I]^\top$, we define the auxiliary variables $u^x := x^k - \hat{v}^{x,k}_y - \alpha_k \mathcal{A}_{\mathcal{S}}^{x\top} \boldsymbol{\xi}^x$ and $u^y := y^k - \hat{v}^{y,k}_x - \alpha_k \mathcal{A}_{\mathcal{S}}^{y\top} \boldsymbol{\xi}^y$. The $i$-th component of the mappings is given by
\begin{equation}
\label{Define_F}
\tag{3.2}
\begin{aligned}
    \mathcal{F}_{y,i}^x(\boldsymbol{\xi}^x) &= \nabla (\phi^x_{y,\kappa(i)})^*(\xi_{y,i}^{x}) - \operatorname{prox}_{\alpha_k \varphi}(u^x), \\
    \mathcal{F}_{x,i}^y(\boldsymbol{\xi}^y) &= \nabla (\phi^y_{x,\kappa(i)})^*(\xi_{x,i}^{y}) - \operatorname{prox}_{\alpha_k \psi}(u^y).
\end{aligned}
\end{equation}
The Newton direction $d = (d^x, d^y)$ is obtained by solving
\begin{equation}
\label{newton_step}
\tag{3.3}
\mathcal{W}_x d^x = -\mathcal{F}^x_y(\boldsymbol{\xi}^x) \quad \text{and} \quad \mathcal{W}_y d^y = -\mathcal{F}^y_x(\boldsymbol{\xi}^y),
\end{equation}
where $\mathcal{W}_x \in \hat{\partial}\mathcal{F}^x_y$ and $\mathcal{W}_y \in \hat{\partial}\mathcal{F}^y_x$ are elements of the surrogate differentials
\begin{equation}
\label{surrogate_differential}
\tag{3.4}
\begin{split}
    \hat{\partial} \mathcal{F}^x_y(\boldsymbol{\xi}^x) &= \left\{ \mathbf{H}^x + \alpha b \mathcal{A}^x_{\mathcal{S}} U^x (\mathcal{A}^x_{\mathcal{S}})^\top \mid U^x \in \partial \operatorname{prox}_{\alpha \varphi}(u^x), \mathbf{H}^x \in \mathcal{H}^x(\boldsymbol{\xi}^x) \right\}, \\
    \hat{\partial} \mathcal{F}^y_x(\boldsymbol{\xi}^y) &= \left\{ \mathbf{H}^y + \alpha b \mathcal{A}^y_{\mathcal{S}} U^y (\mathcal{A}^y_{\mathcal{S}})^\top \mid U^y \in \partial \operatorname{prox}_{\alpha \psi}(u^y), \mathbf{H}^y \in \mathcal{H}^y(\boldsymbol{\xi}^y) \right\}.
\end{split}
\end{equation}
Here, $\mathcal{H}^x(\boldsymbol{\xi}^x)$ denotes the set of block-diagonal matrices $\operatorname{blkdiag}(H_1, \ldots, H_b)$ with $H_i \in \partial (\nabla (\phi^x_{y,\kappa(i)})^*)(\xi^x_{y,i})$, and similarly for $\mathcal{H}^y(\boldsymbol{\xi}^y)$.

We first present several basic properties of the operators and functions involved in the Newton step \eqref{newton_step}.

\begin{prop} \cite[Lem.3.3.5]{milzarek2016numerical}
\label{semi_definite}
Let \(\alpha > 0\) and \(x \in \R^n\),\(y \in \R^m\) be given. Each element \(U^x \in \partial \operatorname{prox}_{\alpha \varphi}(x)\) 
is a symmetric and positive semidefinite \(n \times n\) matrix. Each element \(U^y \in \partial \operatorname{prox}_{\alpha \psi}(y)\) 
is a symmetric and positive semidefinite \(m \times m\) matrix.
\end{prop}

\begin{prop} \cite[Prop.~4.2]{milzarek2024semismooth}
\label{prop:4.2}
\noindent Suppose that the conditions \ref{essentially_differentiable} and \ref{asumme:conjugate_semidef} are satisfied. 
Let \(\mathcal{S}\) be a \(b\)-tuple of elements from \([N]\), and let \(\alpha > 0\), 
\(x, v^x_y \in \R^n\), and \(y, v^y_x \in \R^m\) be given. 
Then, the function \(\mathcal{F}^x_y\) is semismooth on \(\mathcal{D}^x\) with respect to 
\(\hat{\partial}\mathcal{F}_y^x\), and the function \(\mathcal{F}^y_x\) is semismooth on \(\mathcal{D}^y\) 
with respect to \(\hat{\partial}\mathcal{F}_x^y\). If \ref{A5} and \ref{A6} hold, then
\(\mathcal{F}_y^x\) and \(\mathcal{F}_x^y\) are \(\nu\)-order semismooth with respect to 
\(\hat{\partial}\mathcal{F}_y^x\) and \(\hat{\partial}\mathcal{F}_x^y\), respectively.
\end{prop}

In the following, we demonstrate that the functions \(\mathcal{F}^x_y\) and \(\mathcal{F}^y_x\) can be interpreted as gradient mappings. 
Consequently, finding a root of \(\mathcal{F}^x_y\) and \(\mathcal{F}^y_x\) is equivalent to identifying a stationary point.

\begin{prop}
    \label{eq:prop_4.3}
Let the assumptions $\ref{asumme:L-smooth}$ and $\ref{essentially_differentiable}$ hold. Let $\alpha > 0$ and $x, d^x \in \R^n,~$and $y, d^y \in \R^m$ be given and 
let $\mathcal{S}$ be a $b$-tuple of elements of $[N]$. For $\xi^x_y \in \mathcal{D}^x,\xi^y_x \in \mathcal{D}^y $, we define
\begin{align*}
\mathcal{I}^x_y(\xi^x_y) := \sum_{i=1}^b(\phi^x_{y,k(i)})^*(\xi^x_{y,i}) + \frac{b}{2\alpha} \| z^x_y(\xi^x_y) \|^2 
- \frac{b}{\alpha} \operatorname{env}_{\alpha \varphi}(z^x_y(\xi^x_y)),
 \end{align*}
where $z^x_y(\xi^x_y) := x - \frac{\alpha}{b}\sum_{i=1}^b\xi^x_{y,i} - \hat{v}^x_y$. The function $\mathcal{I}^y_x(\boldsymbol{\xi}^y)$ is defined analogously.
Then, $\mathcal{I}^x_y ,\mathcal{I}^y_x $ are $\mu^*_x$-strongly convex, $\mu^*_y$-strongly convex  on the 
set $\mathcal{E}_x := \prod_{i=1}^b \operatorname{dom}(\phi^x_{y,k(i)})^*$, 
$\mathcal{E}_y := \prod_{i=1}^b \operatorname{dom}(\phi^y_{x,k(i)})^*$
respectively,and we have
 \[\nabla_{\xi^x_y} \mathcal{I}^x_y(\xi^x_y) = \mathcal{F}^x_y(\xi^x_y),\quad \nabla_{\xi_x^y} \mathcal{I}^y_x(\xi^y_x) = \mathcal{F}^y_x(\xi^y_x) \]
 for all $\xi^x_y \in \mathcal{D}^x$,$\xi^y_x \in \mathcal{D}^y$,
 where $\mathcal{F}^x_y$,$\mathcal{F}^y_x$ are defined in \eqref{Define_F}.
\end{prop}

\begin{proof}
 For every $\xi^x_y \in \mathcal{D}^x$ and $i \in [b]$, we have 
 \[
 \frac{\partial z^x_y}{\partial \xi^x_{y,i}}(\xi^x_y) = -\frac{\alpha}{b},\quad  \frac{\partial z^y_x}{\partial \xi^y_{x,i}}(\xi^y_x) = -\frac{\alpha}{b}
 \]
Hence,
\begin{align*}
\nabla_{\xi_{y,i}^x} \mathcal{I}^x_y(\xi^x_y) 
&= \nabla_{\xi_{y,i}^x} (\phi^x_y)_{k(i)}^*(\xi^x_{y,i}) 
-z_y^x(\xi^x_y)+ (z_y^x(\xi^x_y)-\prox_{\alpha \varphi}z^x_y(\xi^x_y)) \\
&= \nabla_{\xi_{y,i}^x} (\phi^x_y)_{k(i)}^*(\xi^x_{y,i}) -\prox_{\alpha \varphi}z_y^x(\xi^x_y)=\mathcal{F}^x_{y,i}(\xi^x_y)
\end{align*}
Similarly,we have
\(
\nabla_{\xi_{x,i}^y} \mathcal{I}^y_x(\xi^y_x)=\nabla_{\xi_{x,i}^y} (\phi^y_x)_{k(i)}^*(\xi^y_{x,i})
 -\prox_{\alpha \psi}z^y(\xi^y_x)=\mathcal{F}^y_{x,i}(\xi^y_x).
\)

For convexity, Assumption \ref{asumme:L-smooth} implies that $\sum (\phi^x_{y,\kappa(i)})^*$ is $\mu^*_x$-strongly convex. For the remaining term, we apply Moreau's identity \cite[Thm. 6.67]{beck2017first}
\begin{equation*}
    \frac{b}{2\alpha} \| z^x_y \|^2 - \frac{b}{\alpha} \operatorname{env}_{\alpha \varphi}(z^x_y) = \alpha b \operatorname{env}_{\alpha^{-1} \varphi^*}(z_y^x/\alpha)
\end{equation*}
Since the Moreau envelope of a proper closed convex function is convex \cite[Thm. 6.55]{beck2017first} and convexity is preserved under affine composition ($u^x$ is affine in $\boldsymbol{\xi}^x$), the second term is convex. Thus, $\mathcal{I}^x_y$ is $\mu^*_x$-strongly convex as the sum of a strongly convex function and a convex function.
Similarly, $\mathcal{I}_x^y$ is $\mu_y^*$-strongly convex on $\mathcal{E}_y$.
\end{proof}

In \hyperref[alg:ssn_generic]{Algorithm \ref{alg:ssn_generic}},
 we formulate a globalized semismooth 
Newton method for solving the nonsmooth system \eqref{semi_smooth_equation}.
 Specifically, the result in \hyperref[eq:prop_4.3]{ Proposition \ref{eq:prop_4.3}} 
 enables us to measure descent properties of a semismooth Newton step
  using $\mathcal{I}^x_y,\mathcal{I}^y_x $ and to apply Armijo line search-based globalization techniques. Based on the results on SC1 minimization
   (cf. \cite{zhao2010newton,facchinei1995minimization,j2016proximal}), we obtain the following convergence result.

\begin{thm}\label{thm:sub_conver}
Let the assumptions \ref{asumme:L-smooth}-\ref{asumme:conjugate_semidef} be satisfied and let the sequence $\{\xi^{x,k+1}_y\}$ and $\{\xi^{y,k+1}_x\}$ be 
generated by \hyperref[alg:ssn_generic]{Algorithm \ref{alg:ssn_generic}}.
 Then, $\{\xi^{x,k+1}_y\}$ and $\{\xi^{y,k+1}_x\}$ are converges q-superlinearly to the unique solution $\hat{\xi}^x_y \in \mathcal{D}^x$ and
  $\hat{\xi}^y_x \in \mathcal{D}^y$ of \eqref{semi_smooth_equation} respectively, i.e.,
\begin{align*}
    &\|\xi^{x,k+1}_y - \hat{\xi}^x_y\| = o(\|\xi^{x,k}_y - \hat{\xi}^x_y\|), \quad \text{as } k \rightarrow \infty,\\
    &\|\xi^{y,k+1}_x - \hat{\xi}^y_x\| = o(\|\xi^{y,k}_x - \hat{\xi}^y_x\|), \quad \text{as } k \rightarrow \infty.
\end{align*}
Moreover, under \ref{A5} and \ref{A6}, we obtain
\begin{align*}
&\|\xi^{x,k+1}_y - \hat{\xi}^x_y\| = O(\|\xi^{x,k}_y - \hat{\xi}^x_y\|^{1+\min \{\tau,\nu\}}), \quad \text{for all } k \text{ sufficiently large.}\\
&\|\xi^{y,k+1}_x - \hat{\xi}^y_x\| = O(\|\xi^{y,k}_x - \hat{\xi}^y_x\|^{1+\min \{\tau,\nu\}}), \quad \text{for all } k \text{ sufficiently large.}
\end{align*}
\end{thm}

\begin{proof}
By construction , we have $\{\xi^{x,k+1}_y\}\subseteq D^x$ and set $D^x$ is open , and 
$\{\xi^{y,k+1}_x\}\subseteq D^y$ and set $D^y$ is open. \hyperref[eq:prop_4.3]{Proposition \ref{eq:prop_4.3}}
  and \ref{essentially_differentiable}
 imply that $\mathcal{I}^x_y,\mathcal{I}^y_x $ are $\mu^*_x$-strongly convex and $\mu^*_y$-strongly convex
  on $D^x$ and $D^y$ respectively and essentially differentiable. 
  Hence, $\mathcal{I}^x_y,\mathcal{I}^y_x $ have a unique minimizer $\hat{\xi}^x_y \in D_x $,$\hat{\xi}^y_x \in D_y$ 
  respectively. $\forall \xi_y^x \in D_x ,\xi_x^y \in D_y $, the matrices $W_x(\xi_y^x) \in \hat{\partial} F_y^x(\xi_y^x)$
 and $W_y(\xi_x^y) \in\hat{\partial} F^y_x(\xi^y_x)$ are positive definte by \eqref{eq:iteration_form} and 
 \hyperref[semi_definite]{Proposition \ref{semi_definite}}.
 Using standard arguments ( see \cite{zhao2010newton} Thm3.4 and \cite{li2018highly} Thm3.6), it can be shown that sequence $\{\xi_y^{x,k}\}
,\{\xi_x^{y,k}\}$ generated by \hyperref[alg:ssn_generic]{Algorithm \ref{alg:ssn_generic}} converges to $\hat{\xi}^x_y,\hat{\xi}^y_x.$
Similarly, under \ref{asumme:L-smooth}-\ref{asumme:conjugate_semidef}, we conclude from equation (67) in the proof of [\cite{zhao2010newton}, Thm. 3.5] that 
\[\|\xi^{x,j}_y+d_j^x-\hat{\xi}^x_y\|\leq o(\|\xi^{x,j}_y-\hat{\xi}^x_y\|),
\|\xi^{y,j}_x+d_j^y-\hat{\xi}^y_x\|\leq o(\|\xi^{y,j}_x-\hat{\xi}^y_x\|)\] holds for all j sufficiently large. 
If assumptions \ref{A5} and \ref{A6} are satisfied, then we have  
\(\|\xi_y^{x,j}+d_j^x-\hat{\xi}^x_y\|\leq O(\|\xi^{x,j}_y-\hat{\xi}^x_y\|^{1+\min\{\tau,v\}}),\quad
\|\xi_x^{y,j}+d_j^y-\hat{\xi}^y_x\|\leq O(\|\xi^{y,j}_y-\hat{\xi}^y_x\|^{1+\min\{\tau,v\}}).\)
Finally, let us show that in a neiborhood of the limit point that unit step size is accepted by the armijo line search.
Setting \(\tilde{W}_{x,j}:= W_x+\eta_j^x I,\tilde{W}_{y,j}:= W_y+\eta_j^y I\)  and \(F_y^x(\xi_y^{x,j}) \to 0,F_y^x(\xi_y^{x,j}) \to 0\)
we can infer
\(
\|d_j^x\|=\|\tilde{W}^{-1}_{x,j}(r_j^x-F_y^x(\xi_y^{x,i}))\|
\leq \|\tilde{W}^{-1}_{x,j}\|(\|r_j^x\|+\|F_y^x(\xi_y^{x,j})\|)
\leq  2\lambda_{\min}(\tilde{W}_{x,j})^{-1} \|F_y^x(\xi_y^{x,j})\|,
\)
Similarly, we have
\begin{small}\(
\|d_j^y\|\leq  2\lambda_{\min}(\tilde{W}_{y,j})^{-1} \|F_x^y(\xi_x^{y,j})\|
\)\end{small}
for all  j  sufficiently large. Thus, we have
\begin{align*}
-\frac{\left\langle F^x_y(\xi_y^{x,j}), \bm{d}_{j}^x\right\rangle}{\left\|\bm{d}^{x}_j\right\|^{2}} 
\geq \frac{\lambda_{\min}\left(\widetilde{\mathcal{W}}_{x,j}\right)^{2}}{4} 
\frac{\left\langle - F^x_y(\xi^{x,j}_y), \bm{d}^{x}_j\right\rangle}{\left\|F_y^x \left(\xi^{x,j}_y\right)\right\|^{2}} 
\geq \frac{\lambda_{\min}\left(\widetilde{\mathcal{W}}_{x,j}\right)^{2}}{4 \lambda_{\max}\left(\widetilde{\mathcal{W}}_{x,j}\right)}, 
\end{align*}
Similarly, we have
\begin{small}
\(
-\frac{\left\langle F^y_x(\xi_x^{y,j}), \bm{d}_{j}^y\right\rangle}{\left\|\bm{d}^{y}_j\right\|^{2}} 
\geq \frac{\lambda_{\min}\left(\widetilde{\mathcal{W}}_{y,j}\right)^{2}}{4 \lambda_{\max}\left(\widetilde{\mathcal{W}}_{y,j}\right)}, 
\)\end{small}
\noindent where the second inequality comes from [\cite{zhao2010newton} Prop.3.3] . Due to strong convexity, there exists $\tilde{\rho}>0$ such that 
$$\frac{\lambda_{\min}\left(\widetilde{\mathcal{W}}_{x,j}\right)^{2}}{4\lambda_{\max}\left(\widetilde{\mathcal{W}}_{x,j}\right)} \geq \tilde{\rho}>0, \quad
\frac{\lambda_{\min}\left(\widetilde{\mathcal{W}}_{y,j}\right)^{2}}{4\lambda_{\max}\left(\widetilde{\mathcal{W}}_{y,j}\right)} \geq \tilde{\rho}>0$$
for all $j$. Thanks to [\cite{facchinei1995minimization} Thm3.3], $\beta^x_{j}=\beta^y_j=1$ then fulfills the Armijo condition for $j$ sufficiently large which concludes the proof.
\end{proof}

\begin{algorithm}[h]
\caption{Semismooth Newton Method for Subproblems}
\label{alg:ssn_generic}
\small 
\begin{algorithmic} 
\Require 
Target $u \in \{x, y\}$ with fixed $v \in \{x,y\} \setminus \{u\}$; 
variance reduction vector $\hat{v}_v^u$; 
batch $\mathcal{S} \subseteq [N]$; 
step size $\alpha > 0$, tolerance $\varepsilon_{sub}$.
\State \textbf{Initialize:} $\boldsymbol{\xi}^{u,0}_v \in \mathcal{D}^u$ with $\xi_{v,i}^{u,0} \in \mathcal{D}^u_i ,~\forall i \in \{1,...,b\}.$
\State \textbf{Parameters:} $\hat{\gamma} \in (0,0.5), \eta, \rho, \tau_1, \tau_2 \in (0,1), \tau \in (0,1]$.
\State Set $j = 0$. Define current objective $\mathcal{I}(\cdot) := \mathcal{I}^u_v(\cdot)$ and mapping $\mathcal{F}(\cdot) := \mathcal{F}^u_v(\cdot)$.
\While{$\|\mathcal{F}(\boldsymbol{\xi}^j)\| > \varepsilon_{sub}$} 
    \State \textbf{(Newton Direction)} Compute $\mathcal{F}_j := \mathcal{F}(\boldsymbol{\xi}^j)$ and select $\mathcal{W}_j \in \hat{\partial}\mathcal{F}(\boldsymbol{\xi}^j)$.    
    \State Set $\eta_j := \tau_1 \min\{\tau_2, \|\mathcal{F}_j\|\}$. Solve
     $$(\mathcal{W}_j + \eta_j I)d^j = -\mathcal{F}_j$$ 
    \State via conjugate gradient method such that 
     \( \| (\mathcal{W}_j + \eta_j I)d^j + \mathcal{F}_j \| \leq \min\{\eta_j, \|\mathcal{F}_j\|^{1+\tau}\}. \)        
    \State \textbf{(Line Search)} Find smallest integer $\ell_j \ge 0$ satisfying     
     \[ \mathcal{I}(\boldsymbol{\xi}^j + \rho^{\ell_j}d^j) \leq \mathcal{I}(\boldsymbol{\xi}^j) + \hat{\gamma}\rho^{\ell_j}\langle\nabla \mathcal{I}(\boldsymbol{\xi}^j), d^j \rangle \]
    \Statex \hspace*{1em} with $\boldsymbol{\xi}^j + \rho^{\ell_j} d^j \in \mathcal{D}^u$.     
    \State \textbf{(Update)} Set $\boldsymbol{\xi}^{j+1} = \boldsymbol{\xi}^j + \rho^{\ell_j} d^j$ and $j \gets j + 1$.
\EndWhile
\State \textbf{Return} $\boldsymbol{\xi}^j$.
\end{algorithmic}
\end{algorithm}

\begin{rem}
The tolerance $\varepsilon_{\mathrm{sub}}$ in each call of \hyperref[alg:ssn_generic]{ Algorithm~\ref{alg:ssn_generic}} corresponds to the tolerance $\varepsilon_k$ in \hyperref[alg:SNmMSPP]{Algorithm~\ref{alg:SNmMSPP}}.
\end{rem}

\section{Controlling the Inexactness of the Update}\label{sec:controllinexact}

The strong convexity of the objective functions plays a crucial role in controlling the approximation error. Specifically, $\mathcal{I}^x_y$ and $\mathcal{I}^y_x$ are $\mu_x^*$- and $\mu_y^*$-strongly convex, respectively. Let $\hat{\xi}^x_y$ and $\hat{\xi}^y_x$ denote their unique minimizers.
Since the gradients of strongly convex functions are strongly monotone, and noting that the gradients vanish at the optimal solutions, we have
\[
\mu_x^* \| \xi^x_y - \hat{\xi}^x_y \|^2 \leq \langle \nabla \mathcal{I}^x_y(\xi^x_y), \xi^x_y - \hat{\xi}^x_y \rangle \leq \|\nabla \mathcal{I}^x_y(\xi^x_y)\| \|\xi^x_y - \hat{\xi}^x_y\|, \quad \forall \xi^x_y \in \mathcal{D}_x,
\]
and similarly for $\mathcal{I}^y_x$. This immediately yields the error bounds
\begin{equation}\label{eq:xi_gap}
\| \xi^x_y - \hat{\xi}^x_y\| \leq \frac{1}{\mu_x^*} \| \nabla \mathcal{I}^x_y(\xi^x_y) \|, 
\quad 
\| \xi^y_x - \hat{\xi}^y_x\| \leq \frac{1}{\mu_y^*} \| \nabla \mathcal{I}^y_x(\xi^y_x) \|. \tag{4.1}
\end{equation}
Thus, the stopping criterion $\|\nabla \mathcal{I}(\cdot)\| \leq \varepsilon_{\text{sub}}$ in \hyperref[alg:ssn_generic]{Algorithm~\ref{alg:ssn_generic}} effectively limits the distance between the approximate and exact solutions.

\begin{thm} \label{prop:control_inexact}
Let the assumptions of \hyperref[eq:prop_4.3]{Proposition \ref{eq:prop_4.3}} hold. Let $\hat{\xi}^x_y$ and $\hat{\xi}^y_x$ be the unique minimizers of $\mathcal{I}^x_y$ and $\mathcal{I}^y_x$, respectively. 
Suppose the algorithm terminates with $\|\nabla \mathcal{I}\| \leq \varepsilon_{\text{sub}}$ and returns approximations $\xi^x_y, \xi^y_x$. 
Let $x^+, y^+$ be the computed updates and $\hat{x}^+, \hat{y}^+$ be the exact updates (using $\hat{\xi}$ instead of $\xi$). Then
\(
    \| \hat{\xi}_{x}^{y} - \xi_{x}^{y} \| \leq \frac{\varepsilon_{\text{sub}}}{\mu_y^*}, 
    \| \hat{\xi}_{y}^{x} - \xi_{y}^{x} \| \leq \frac{\varepsilon_{\text{sub}}}{\mu_x^*},
\)
and the propagation errors are bounded by
\(
    \| x^+ - \hat{x}^+ \| \leq \frac{\alpha}{\mu_x^* b} \varepsilon_{\text{sub}}, \quad 
    \| y^+ - \hat{y}^+ \| \leq  \frac{\alpha}{\mu_y^* b} \varepsilon_{\text{sub}}.
\)
\end{thm}

\begin{proof}
The bounds on $\|\hat{\xi} - \xi\|$ follow directly from \eqref{eq:xi_gap}. 
For the update steps, since $\operatorname{prox}_{\alpha \varphi}$ is non-expansive, we have
\begin{align*}
\| x^+ - \hat{x}^+ \| 
&= \left\| \operatorname{prox}_{\alpha \varphi} \left( x - 
\frac{\alpha}{b} \sum_{i=1}^b\hat{\xi}_{y,i}^{x}-\hat{v}^{x}_y \right) - 
\operatorname{prox}_{\alpha \varphi} \left( x - \frac{\alpha}{b} \sum_{i=1}^b \xi_{y,i}^{x} - \hat{v}^x_y \right) \right\| \\
&\leq \left\| \frac{\alpha}{b} \sum_{i=1}^b (\hat{\xi}_{y,i}^{x} - \xi_{y,i}^{x}) \right\| \\
&\leq \frac{\alpha}{b} \| \hat{\xi}_{y}^{x} - \xi_{y}^{x} \| \\
&\leq \frac{\alpha}{\mu_x^* b} \varepsilon_{\text{sub}}.
\end{align*}
The bound for $\| y^+ - \hat{y}^+ \|$ follows analogously.
\end{proof}

\section{Convergence Analysis}\label{sec:convergence_analysis}

In this section, we establish q-linear convergence of \hyperref[alg:SNmMSPP]{Algorithm \ref{alg:SNmMSPP}} . We derive 
- similar to Thm.3.1 in \cite{yang2015sdpnal+} - convergence in terms of the objective function. We suppose that in iteration $s$ of the outer
 and iteration $k$ of the inner loop of \hyperref[alg:SNmMSPP]{Algorithm \ref{alg:SNmMSPP}}, the tolerances $\varepsilon_k^s$ satisfy the bound
\begin{equation*}
    \tag{5.1}
\label{tolerance_bound}
\varepsilon_k^s \leq \delta_s \| F_{\text{y,nat}}^x(\tilde{x}^s) \| \quad \text{and} \quad \varepsilon_k^s \leq \delta_s \| F_{x,\text{nat}}^y(\tilde{y}^s) \|
\end{equation*}
for all $k \in \{0, \ldots, m-1\}$, $s \in \mathbb{N}$, and for some sequence $\mathbb{R}_+ \ni \delta_s \to 0$. 
Since $\nabla f(\tilde{x}^s)$ 
is known, $\| F_{\text{y,nat}}^x(\tilde{x}^s) \|$ and $\| F_{x,\text{nat}}^y(\tilde{y}^s) \|$ can be computed without additional costs.

\begin{assump}[Unified Error Bound for Dual Variables]
\label{ass:unified_error_bound}
Suppose there exists a constant $\kappa > 0$ such that for the iterates $x^{s,k}, y^{s,k}, \lambda^{s,k}$
generated by \hyperref[alg:SNmMSPP]{Algorithm \ref{alg:SNmMSPP}}, 
the following inequalities hold simultaneously
\begin{equation*}
\begin{aligned}
\|\lambda^{s,k} - \lambda^*\| &\leq \kappa \|A x^{s,k} + B y^{s,k} + c\|, \quad \forall \, 0 \leq s \leq S, \, 0 \leq k < m,
\end{aligned}
\end{equation*}
\end{assump}

\begin{Lemma}
\label{lem:young_inq_gene}
For any \(a, b, c \in \R^n\) and \(\rho \in (0,1)\), the following inequality holds
\[
\|a - b\|^2 \geq (1 - \rho)\|c-b\|^2 + \left(1 - \frac{1}{\rho}\right)\|a-c\|^2.
\]
\end{Lemma}
\begin{proof}
Expanding the squared norm and applying Young's inequality $2\langle u, v \rangle \geq -\frac{1}{\rho}\|u\|^2 - \rho\|v\|^2$ with $u=a-c$ and $v=c-b$, we obtain\begin{small}
\(
\|a - b\|^2 = \|(a-c) + (c-b)\|^2 \geq \|a-c\|^2 + \|c-b\|^2 - \frac{1}{\rho}\|a-c\|^2 - \rho\|c-b\|^2 
= \left(1 - \frac{1}{\rho}\right)\|a-c\|^2 + (1 - \rho)\|c-b\|^2.
\)\end{small}
\end{proof}

\begin{thm}
    \label{thm:opt_q-linear_convergence}
 Let \hyperref[eq:assume_2.1]{Assumption \ref{eq:assume_2.1}}, \hyperref[eq:assume_2.2]{Assumption \ref{eq:assume_2.2}} and
 \hyperref[ass:unified_error_bound]{Assumption \ref{ass:unified_error_bound}}be satisfied
  and let $L(x,y,\lambda)$  be $\mu_x$-strongly convex w.r.t x and $\mu_y$-strongly concave w.r.t y .
  We set $\mu_{min}:=\{\mu_x,\mu_y\},\mu_{max}=\max\{\mu_x,\mu_y\}.$

Consider \hyperref[alg:SNmMSPP]{Algorithm \ref{alg:SNmMSPP}} with $S = \infty$ and Option I, using constant step sizes
 $\alpha_k^s = \alpha > 0$ and constant batch sizes $b_k^s = b$. Assume that
 \begin{align*}
\alpha < \min \left\{ 
\left( 
\bar{L}_{\phi,y} 
\right. \right. & + \frac{\sqrt{2m(m-1)}\,\bar{L}_{\phi,y}}{\sqrt{b}} + \frac{\sqrt{m(m-1)}\,\bar{L}_{\phi,x}}{\sqrt{2b}} 
\left. \vphantom{\frac{\sqrt{2m(m-1)}}{\sqrt{b}}} \right)^{-1}, \\
& \left( 
\bar{L}_{\phi,x} 
+ \frac{\sqrt{2m(m-1)}\,\bar{L}_{\phi,x}}{\sqrt{b}} + \frac{\sqrt{m(m-1)}\,\bar{L}_{\phi,y}}{\sqrt{2b}} 
\left. \vphantom{\frac{\sqrt{2m(m-1)}}{\sqrt{b}}} \right)^{-1} 
\right\}
\end{align*}
 and let \eqref{tolerance_bound} hold for a given sequence $\{\delta_s\}$ satisfying 
 $$\delta_s <\min \left\{ \frac{1+2\alpha \mu_{min}}{1+\alpha \mu_{max}} ,
 \frac{2\alpha \mu_x}{1+\alpha \mu_x},\frac{2\alpha \mu_y}{1+\alpha \mu_y} \right\}$$ 
 for all $s$. 
 Then, the iterates $\{ \tilde{x}^s,\tilde{y}^s \}$ converge $q$-linearly in expectation to the unique solution $(x^*,y^*)$ of
  problem \eqref{problem1}, i.e., as $s \to \infty$, we have
\begin{align*}
     & \E\|(\tilde{x}^{s+1},\tilde{y}^{s+1})-(x^*,y^*)\|^2\\
\leq&\left( 1-\frac{ 2\alpha\mu_{min} }{1 + 2\alpha \mu_{min}}+O(\delta_s)\right)\E\|(\tilde{x}^{s},\tilde{y}^{s})-(x^*,y^*)\|^2, \quad \text{as}\quad s \to \infty.
\end{align*}
\noindent Meanwhile, the sequence $\{\tilde{\lambda}^s\}$ also r-linearly converges in expectation to the Lagrange multiplier $\lambda^*$ of problem \eqref{problem1}.

\end{thm}

\begin{proof}
 Fix \( s \in \mathbb{N}_0 \) and let \( k \in \{0, \ldots, m-1\} \) be given.
  Let again \( (\hat{x}^{k+1},\hat{y}^{k+1}, \hat{\xi}^{k+1}) \) denote the pair of exact solutions of (\ref{(3.3)}) and 
  \eqref{equation_1}. Due to 

\[
\begin{cases*}
\tag{5.2}
\label{eq:exact xi}
\hat{\xi}_{y,i}^{x,k+1} = \nabla g_{\kappa(i)}(\hat{x}^{k+1}) + \nabla_x f_{k(i)} (\hat{x}^{k+1},\hat{y}^{k+1}),\\
\hat{\xi}_{x,i}^{y,k+1} =  \nabla h_{\kappa(i)}(\hat{y}^{k+1}) - \nabla_y f_{k(i)} (\hat{x}^{k},\hat{y}^{k+1}) \quad i \in [b],
\end{cases*}
\]
 we have 
\(
\hat{x}^{k+1} = \operatorname{prox}_{\alpha \varphi}\!\left( x^k - \frac{\alpha}{b}\sum_{i\in \mathcal{S}_k}\hat{\xi}^{x,k+1}_{y,i} - \hat{v}^{x,k+1}_y  \right),\) the update for $y$ is obtained similarly.
\(\hat{y}^{k+1} = \operatorname{prox}_{\alpha \psi}\!\left( y^k - \frac{\alpha}{b}\sum_{i\in \mathcal{S}_k}\hat{\xi}^{y,k+1}_{x,i} - \hat{v}^{y,k+1}_x  \right).\)
Furthermore, introducing
\( L^x_k(x,y^{k+1},\lambda^{k}) := L_{\mathcal{S}_k}(x,y^{k+1},\lambda^{k}) + \langle v^{x,k}_y, x - x^k \rangle, \)
\( L^y_k(x^k,y,\lambda^{k}) := L_{\mathcal{S}_k}(x^k,y,\lambda^{k}) + \langle v^{y,k}_x, y - y^k \rangle, \)
 the underlying optimality condition of the proximity operator  implies 
\[
\begin{aligned}
p = \hat{x}^{k+1} \quad &\iff \quad p \in x^k-\alpha_k\bigl(\nabla g_{\mathcal{S}_k}(p) + \nabla_x f_{\mathcal{S}_k}(p, y^{k+1}) + A^\top \lambda^k - \partial \varphi(p) + v^{x,k}_y\bigr) \\
&\iff \quad p = \operatorname{prox}_{\alpha L^x_k}\!\left( x^k \right), 
\end{aligned}
\]
Similarly, $q = \hat{y}^{k+1} \iff  q = \operatorname{prox}_{\alpha L^y_k}\!\left( y^k \right)$.
Moreover, using \hyperref[eq:assume_2.1]{Assumption \ref{eq:assume_2.1}}, the mapping
\(
x \mapsto \phi^x_{y,S_k}(x)=g_{\mathcal{S}_k}(x)+f_{\mathcal{S}_k}(x,y)=\frac{1}{b}\sum_{i\in \mathcal{S}_k}(g_{i}(x)+f_{i}(x,y))\)   is $\mu_x$-strongly convex,
\(y \mapsto \phi^y_{x,S_k}(y)\)  is $\mu_y$-strongly convex.

Hence,setting 
\[
\begin{cases*}
    \Gamma^{x,k+1}(x,y,\lambda)=g_{\mathcal{S}_k}(x)+f_{\mathcal{S}_k}(x,y)+\varphi(x)+(\lambda)^\top Ax+\langle v^{x,k}_y,x-x^k\rangle,\\
    \Gamma^{y,k+1}(x,y,\lambda)=-h_{\mathcal{S}_k}(y)+f_{\mathcal{S}_k}(x,y)-\psi(y)+(\lambda)^\top By+\langle v^{y,k}_x,y-y^k\rangle,
\end{cases*}
\]
then $x \mapsto \Gamma^{x,k+1}(x,y^{k+1},\lambda^k)+\frac{1}{2\alpha}\|x-x^k\|^2$ is $(\mu_x+\frac{1}{\alpha})$-strongly convex,
and $y \mapsto \Gamma^{y,k+1}(x^k,y,\lambda^k)-\frac{1}{2\alpha}\|y-y^k\|^2$ is $(\mu_y+\frac{1}{\alpha})$-strongly concave,
and due to 
\(
\hat{x}^{k+1}=\arg\underset{x}{\min} \{\Gamma^{x,k+1}(x,y^{k+1},\lambda^k)+\frac{1}{2\alpha}\|x-x^k\|^2\},
\hat{y}^{k+1}=\arg\underset{y}{\max} \{\Gamma^{y,k+1}(x^{k},y,\lambda^k)-\frac{1}{2\alpha}\|y-y^k\|^2\},
\)
it follows
\begin{subequations}
\begin{align}
\Gamma^{x,k+1}(x,y^{k+1},\lambda^k)+\frac{1}{2\alpha}\|x-x^k\|^2 \ge {} & \Gamma^{x,k+1}(\hat{x}^{k+1},y^{k+1},\lambda^k)+\frac{1}{2\alpha}\|x^{k+1}-x^k\|^2 \notag \\
& +\frac{1}{2\alpha}\left(\mu_x+\frac{1}{\alpha}\right)\|x-\hat{x}^{k+1}\|^2, \tag{5.3a} \label{eq:Gamma strongly convex}\\
-\Gamma^{y,k+1}(x^k,y,\lambda^k)+\frac{1}{2\alpha}\|y-y^k\|^2 \ge {} & -\Gamma^{y,k+1}(x^{k},\hat{y}^{k+1},\lambda^k)+\frac{1}{2\alpha}\|\hat{y}^{k+1}-y^k\|^2 \notag\\
& +\frac{1}{2\alpha}\left(\mu_y+\frac{1}{\alpha}\right)\|y-\hat{y}^{k+1}\|^2, \tag{5.3b} \label{eq:Gamma strongly concave}
\end{align}
\end{subequations}
Next, combining the optimality condition, the update rule of 
\hyperref[alg:SNmMSPP]{Algorithm \ref{alg:SNmMSPP}}, and (\ref{eq:exact xi}) gives

\begin{align*}
x^{k+1} &\in x^k - \alpha\bigl( \nabla g_{\mathcal{S}_k}(x^{k+1}) + \nabla_x 
f_{\mathcal{S}_k}(x^{k+1}, y^{k+1}) + A^\top \lambda^k + \partial \varphi(x^{k+1}) + v^{x,k}_y \bigr) \\
&= x^k-\alpha H_y^{ x,k+1}-\alpha\partial_x \Gamma ^{x,k+1}(x^{k+1},y^{k+1},\lambda^k).
\end{align*}

Similarly, we have
\(
y^{k+1} \in y^k + \alpha\bigl( -\nabla h_{\mathcal{S}_k}(y^{k+1}) + \nabla_y f_{\mathcal{S}_k}(x^{k}, y^{k+1}) + B^\top \lambda^k - \partial \psi(y^{k+1}) + v^{y,k}_x \bigr) 
= y^k-\alpha H_x^{ y,k+1}+\alpha\partial_y \Gamma ^{y,k+1}(x^{k},y^{k+1},\lambda^k),
\)
Where we let
\begin{equation}
\label{eq:Hxky-definition}
\tag{H.1}
\begin{aligned}
H^{x,k+1}_y &:= (\mathcal{A}_{\mathcal{S}_k}^x)^\top(\xi_y^{x,k+1} - \hat{\xi}_y^{x,k+1}) + \nabla g_{\mathcal{S}_k}(\hat{x}^{k+1}) - \nabla g_{\mathcal{S}_k}(x^{k+1}) \\
&\quad + \nabla_x f_{\mathcal{S}_k}(\hat{x}^{k+1}, \hat{y}^{k+1}) - \nabla_x f_{\mathcal{S}_k}(x^{k+1}, y^{k+1}),\\
H^{y,k+1}_x &:= (\mathcal{A}_{\mathcal{S}_k}^y)^\top(\xi_x^{y,k+1} - \hat{\xi}_x^{y,k+1}) + \nabla h_{\mathcal{S}_k}(\hat{y}^{k+1}) - \nabla h_{\mathcal{S}_k}(y^{k+1}) \\
&\quad + \nabla_y f_{\mathcal{S}_k}(x^{k}, y^{k+1}) - \nabla_y f_{\mathcal{S}_k}(\hat{x}^{k}, \hat{y}^{k+1}).
\end{aligned}
\end{equation}

\noindent This shows that
\[
\begin{cases*}
\label{eq:change partial Gamma}
\tag{5.4}
x^k-x^{k+1}-\alpha H^{x,k+1}_y  \in \alpha \partial_x \Gamma^{x,k+1}(x^{k+1},y^{k+1},\lambda^k)\\
y^{k+1}-y^{k}+\alpha H^{y,k+1}_x  \in \alpha \partial_y \Gamma^{y,k+1}(x^k,y^{k+1},\lambda^k)
\end{cases*}
\]
Thus, due to the strong convexity of ~$ \Gamma^{x,k+1}(x,y^{k+1},\lambda^k)$ and applying\\
 $-\langle a,b \rangle=\frac{1}{2}\|a-b\|^2-\frac{1}{2}\|a\|^2-\frac{1}{2}\|b\|^2$, it follows
 
\begin{align*}
 &\Gamma^{x,k+1}(z_x,y^{k+1},\lambda^k)-\Gamma^{x,k+1}(x^{k+1},y^{k+1},\lambda^k) \tag*{(using \eqref{eq:change partial Gamma})}\\
\ge & \frac{1}{\alpha}\langle x^k-x^{k+1},z_x-x^{k+1}\rangle-\langle H^{x,k+1}_y,z_x-x^{k+1}\rangle+\frac{\mu_x}{2}\|z_x-x^{k+1}\|^2 \\
=&-\langle H^{x,k+1}_y,z_x-x^{k+1}\rangle +\frac{1}{2}(\mu_x+\frac{1}{\alpha})\|z_x-x^{k+1}\|^2-\frac{1}{2\alpha}\|x^k-z_x\|^2\\
&+\frac{1}{2\alpha}\|x^k-x^{k+1}\|^2,
\tag{5.5a}\label{eq:ge Gamma^x}
\end{align*}
 for all $z_x \in dom(\varphi)$ ,due to the strong concavity of ~$ \Gamma^{y,k+1}(x^k,y,\lambda^k)$ and applying
 $-\langle a,b \rangle=-\frac{1}{2}\|a+b\|^2+\frac{1}{2}\|a\|^2+\frac{1}{2}\|b\|^2$, it follows that
 \begin{align*}
&-\Gamma^{y,k+1}(x^k,z_y,\lambda^k)-(-\Gamma^{y,k+1}(x^{k},y^{k+1},\lambda^k))\\
\ge &-\langle H^{y,k+1}_x,z_y-y^{k+1}\rangle +\frac{1}{2}(\mu_y+\frac{1}{\alpha})\|z_y-y^{k+1}\|^2-\frac{1}{2\alpha}\|y^k-z_y\|^2\\
&+\frac{1}{2\alpha}\|y^k-y^{k+1}\|^2,  \tag{5.5b}\label{eq:ge Gamma^y}
\end{align*}
for all $z_y \in dom(\psi)$.
By substituting \( z_x = \hat{x}^{k+1} \) into \(\eqref{eq:ge Gamma^x}\) and then applying this estimate 
to \(\eqref{eq:Gamma strongly convex}\), followed by Young's inequality, we obtain
 \begin{align*}
&\frac{1}{2}(\mu_x+\frac{1}{\alpha})\|x-\hat{x}^{k+1}\|^2-\frac{1}{2\alpha}\|x-x^k\|^2 \\
\leq & \Gamma^{x,k+1}(x,y^{k+1},\lambda^k)-\Gamma^{x,k+1}(\hat{x}^{k+1},y^{k+1},\lambda^k)-\frac{1}{2\alpha}\|\hat{x}^{k+1}-x^k\|^2\\
\leq &\Gamma^{x,k+1}(x,y^{k+1},\lambda^k)-\Gamma^{x,k+1}(x^{k+1},y^{k+1},\lambda^k)+\frac{\alpha}{2} \|H^{x,k+1}_y\|^2\\
&-\frac{1}{2}\mu_x\|\hat{x}^{k+1}-x^{k+1}\|^2-\frac{1}{2\alpha}\|x^k-x^{k+1}\|^2.\tag{5.6a}\label{eq:ge Gamma^x2} 
\end{align*}

Similarly, by substituting \( z_y = \hat{y}^{k+1} \) into \(\eqref{eq:ge Gamma^y}\) and then applying this estimate
 to \(\eqref{eq:Gamma strongly concave}\), followed by Young's inequality, we obtain
 \begin{align*}
&\frac{1}{2}(\mu_y+\frac{1}{\alpha})\|y-\hat{y}^{k+1}\|^2-\frac{1}{2\alpha}\|y-y^k\|^2 \\
\leq&\frac{\alpha}{2}\|H^{y,k+1}_x\|^2-\frac{1}{2}\mu_y\|\hat{y}^{k+1}-y^{k+1}\|^2-\frac{1}{2\alpha}\|y^k-y^{k+1}\|^2
\\
&+[-\Gamma^{y,k+1}(x^k,y,\lambda^k)+\Gamma^{y,k+1}(x^{k},y^{k+1},\lambda^k)].
 \tag{5.6b}\label{eq:ge Gamma^y2}
 \end{align*}

Next, we expand the first term on the right hand side in (\ref{eq:ge Gamma^x2}) as follows
\begin{align*}
&\Gamma^{x,k+1}(x, y^{k+1},\lambda^k) - \Gamma^{x,k+1}(x^{k+1}, y^{k+1},\lambda^k)\\ 
=& L(x, y^{k+1}, \lambda^k) - L(x^{k+1}, y^{k+1}, \lambda^k) + \left( g_{\mathcal{S}_k}(x) - g(x) \right) + \left( f_{\mathcal{S}_k}(x, y^{k+1}) - f(x, y^{k+1}) \right) \\
&\quad + \left( \phi^x_{y^{k+1}, \mathcal{S}_k}(x^k) - \phi^x_{y^{k+1}, \mathcal{S}_k}(x^{k+1}) \right) + \left( g(x^{k+1}) - g(x^k) \right)+ \left( g(x^k) - g_{\mathcal{S}_k}(x^k) \right) \\
&\quad  + \left( f(x^{k+1}, y^{k+1}) - f(x^k, y^{k+1}) \right) + \left( f(x^k, y^{k+1}) - f_{\mathcal{S}_k}(x^k, y^{k+1}) \right) + \langle v^{x,k}_y, x - x^{k+1} \rangle. \\
\tag{5.7a}
\label{eq:ex Gamma^x}
\end{align*}

Similarly, we expand the first term on the right hand side in (\ref{eq:ge Gamma^y2}) as follows
\begin{align*}
&-\Gamma^{y,k+1}(x^k, y,\lambda^k) + \Gamma^{y,k+1}(x^{k}, y^{k+1},\lambda^k)\\
=& L(x^k, y^{k+1}, \lambda^k) - L(x^{k}, y, \lambda^k) + \left( h_{\mathcal{S}_k}(y) - h(y) \right) + \left( f(x^k, y) - f_{\mathcal{S}_k}(x^k, y) \right) \\
&\quad + \left( \phi^y_{x^{k}, \mathcal{S}_k}(y^{k}) - \phi^y_{x^{k}, \mathcal{S}_k}(y^{k+1}) \right) + \left( h(y^{k+1}) - h(y^k) \right)+ \left( h(y^k) - h_{\mathcal{S}_k}(y^k) \right) \\
&\quad  + \left( f_{\mathcal{S}_k}(x^{k}, y^{k}) - f(x^k, y^{k}) \right) + \left( f(x^k, y^{k}) - f(x^k, y^{k+1}) \right) + \langle v^{y,k}_x, y^{k+1} - y \rangle. \\
\tag{5.7b}
\label{eq:ex Gamma^y}
\end{align*}

Combining the convexity of $\phi^x_{y^{k+1}, \mathcal{S}_k}$ and $\phi^y_{x^k, \mathcal{S}_k}$ with the descent lemma (due to the smoothness of $\phi^x_y$ and $\phi^y_x$), we obtain
\begin{subequations}
\label{eq:ex_Gamma_inter}
\begin{align}
&\phi^x_{y^{k+1}, \mathcal{S}_k}(x^k) - \phi^x_{y^{k+1}, \mathcal{S}_k}(x^{k+1}) + \left[ g(x^{k+1}) - g(x^k) \right] + \left[ f(x^{k+1}, y^{k+1}) - f(x^k, y^{k+1}) \right] \notag \\
\leq{} & \langle x^{k+1} - x^k, \nabla g(x^k) - \nabla g_{\mathcal{S}_k}(x^k) + \nabla_x f(x^k, y^{k+1}) - \nabla_x f_{\mathcal{S}_k}(x^k, y^{k+1}) \rangle \notag \\
& + \frac{\bar{L}_{\phi,x}}{2} \|x^{k+1} - x^k\|^2, \tag{5.8a} \label{eq:ex_Gamma^x_inter} \\
&\phi^y_{x^k, \mathcal{S}_k}(y^{k}) - \phi^y_{x^k, \mathcal{S}_k}(y^{k+1}) + \left[ h(y^{k+1}) - h(y^k) \right] + \left[ f(x^k, y^k) - f(x^k, y^{k+1}) \right] \notag \\
\leq{} & \langle y^{k+1} - y^k, \nabla h(y^k) - \nabla h_{\mathcal{S}_k}(y^k) - \nabla_y f(x^k, y^k) + \nabla_y f_{\mathcal{S}_k}(x^k, y^k) \rangle \notag \\
& + \frac{\bar{L}_{\phi,y}}{2} \|y^{k+1} - y^k\|^2. \tag{5.8b} \label{eq:ex_Gamma^y_inter}
\end{align}
\end{subequations}

By applying \hyperref[lem:young_inq_gene]{Lemma \ref{lem:young_inq_gene}} here, we can derive the following results
\[
\begin{cases*}
    \tag{5.9}
    \label{eq:lemma6.1_xy}
    \|\hat{x}^{k+1}-x\|^2 \ge (1-\rho_1^x)\|x^{k+1}-x\|^2+(1-\frac{1}{\rho_1^x})\|\hat{x}^{k+1}-x^{k+1}\|^2,\\
    \|\hat{y}^{k+1}-y\|^2 \ge (1-\rho_1^y)\|y^{k+1}-y\|^2+(1-\frac{1}{\rho_1^y})\|\hat{y}^{k+1}-y^{k+1}\|^2,\\
\end{cases*}
\]
for $\rho_1^x,\rho_1^y \in (0,1) $.By combining \(\eqref{eq:ge Gamma^x2}\), \(\eqref{eq:ex Gamma^x}\), \(\eqref{eq:ex_Gamma^x_inter}\), and \(\eqref{eq:lemma6.1_xy}\),
 we can obtain the following result
\begin{align*}
&\frac{1}{2}\left(\mu_x + \frac{1}{\alpha}\right) \left[(1 - \rho_1^x) \|x^{k+1} - x\|^2 + (1 - \frac{1}{\rho_1^x}) \|\hat{x}^{k+1} - x^{k+1}\|^2\right] - \frac{1}{2\alpha} \|x - x^k\|^2, \\
\leq{} & L(x, y^{k+1}, \lambda^k) - L(x^{k+1}, y^{k+1}, \lambda^k) + (g_{\mathcal{S}_k}(x) - g(x)) + (f_{\mathcal{S}_k}(x, y^{k+1}) - f(x, y^{k+1})) \\
& + (g(x^k) - g_{\mathcal{S}_k}(x^k)) + (f(x^k, y^{k+1}) - f_{\mathcal{S}_k}(x^k, y^{k+1})) + \langle v_y^{x,k}, x - x^{k} \rangle \\
& + \langle x^{k+1} - x^k, \nabla g(x^k) - \nabla g_{\mathcal{S}_k}(x^k) + \nabla_x f(x^k, y^{k+1}) - \nabla_x f_{\mathcal{S}_k}(x^k, y^{k+1}) - v_y^{x,k} \rangle \\
& + \frac{\bar{L}_{\phi ,x}}{2} \|x^{k+1} - x^k\|^2 + \frac{\alpha}{2} \|H_y^{x,k+1}\|^2 - \frac{\mu_x}{2} \|\hat{x}^{k+1} - x^{k+1}\|^2- \frac{1}{2\alpha} \|x^k - x^{k+1}\|^2.
\end{align*}
Similarly, by combining \(\eqref{eq:ge Gamma^y2}\), \(\eqref{eq:ex Gamma^y}\), \(\eqref{eq:ex_Gamma^y_inter}\), and \(\eqref{eq:lemma6.1_xy}\), we can obtain the following result
\begin{align*}
&\frac{1}{2}\left(\mu_y + \frac{1}{\alpha}\right) \left[(1 - \rho_1^y) \|y^{k+1} - y\|^2 + (1 - \frac{1}{\rho_1^x}) \|\hat{y}^{k+1} - y^{k+1}\|^2\right] - \frac{1}{2\alpha} \|y - y^k\|^2 \\
\leq{}&L(x^k, y^{k+1}, \lambda^k) - L(x^{k}, y, \lambda^k) + \left( h_{\mathcal{S}_k}(y) - h(y) \right) + \left( f(x^k, y) - f_{\mathcal{S}_k}(x^k, y) \right) \\
&\quad  + \left( f_{\mathcal{S}_k}(x^{k}, y^{k}) - f(x^k, y^{k}) \right) +\left( h(y^k) - h_{\mathcal{S}_k}(y^k) \right)  + \langle v^{y,k}_x, y^{k} - y \rangle\\
& \quad + \frac{\alpha}{2} \|H_x^{y,k+1}\|^2 - \frac{\mu_y}{2} \|\hat{y}^{k+1} - y^{k+1}\|^2 - \frac{1}{2\alpha} \|y^k - y^{k+1}\|^2+ \frac{\bar{L}_{\phi,y}}{2} \|y^{k+1} - y^k\|^2 \\
&\quad+ \langle y^{k+1} - y^k, \nabla h(y^k) - \nabla h_{\mathcal{S}_k}(y^k) - \nabla_y f(x^k, y^k) + \nabla_y f_{\mathcal{S}_k}(x^k, y^k) +v^{y,k} \rangle.\notag 
\end{align*}

Next, we perform some manipulations on the above inequalities. Setting 
\(e^{x,k}_y(x):=g_{\mathcal{S}_k}(x)-g(x)+g(x^k)-g_{\mathcal{S}_k}(x^k)+f_{\mathcal{S}_k}(x,y^{k+1})-f(x,y^{k+1})+f(x^k,y^{k+1})
-f_{\mathcal{S}_k}(x^k,y^{k+1})+\langle v^{x,k}_y,x-x^k\rangle,\)
\(e^{y,k}_x(y):=h_{\mathcal{S}_k}(y)-h(y)+h(y^k)-h_{\mathcal{S}_k}(y^k)+f_{\mathcal{S}_k}(x^k,y)-f(x^k,y)+f_{\mathcal{S}_k}(x^k,y^{k})
-f(x^k,y^{k})+\langle v^{y,k}_x,y^k-y\rangle,\)
~then, applying Young's inequality once more, we have
\begin{align*}
    &\langle x^{k+1} - x^k, \nabla g(x^k) - \nabla g_{\mathcal{S}_k}(x^k) + \nabla_x f(x^k, y^{k+1}) - \nabla_x f_{\mathcal{S}_k}(x^k, y^{k+1}) - v_y^{x,k} \rangle\\
    \leq & \frac{1}{2\rho_2^x}\|\nabla g(x^k) - \nabla g_{\mathcal{S}_k}(x^k) + \nabla_x f(x^k, y^{k+1}) - \nabla_x f_{\mathcal{S}_k}(x^k, y^{k+1}) - v_y^{x,k}\|^2\\
    &+ \frac{\rho_2^x}{2}\|x^{k+1}-x^k\|^2, \text{ for all }~x\in \dom(\varphi),\rho_2^x\in(0,1).
\end{align*}
Similarly, we have
\(\langle y^{k+1} - y^k, \nabla h(y^k) - \nabla h_{\mathcal{S}_k}(y^k) -\nabla_y f(x^k, y^{k}) + \nabla_y f_{\mathcal{S}_k}(x^k, y^{k}) +v^{y,k}_x \rangle
    \leq  \frac{1}{2\rho_2^y}\|\nabla h(y^k) - \nabla h_{\mathcal{S}_k}(y^k) -\nabla_y f(x^k, y^{k}) + \nabla_y f_{\mathcal{S}_k}(x^k, y^{k}) +v^{y,k}_x\|^2
    + \frac{\rho_2^y}{2}\|y^{k+1}-y^k\|^2, \text{ for all }~y\in \dom(\psi),\ \rho_2^y\in(0,1).
\)
Hence, we obtian
\begin{align*}
\tag{C.2a}
\label{eq:important_inequality_x}
&(1+\alpha \mu_x)(1-\rho_1^x)\|x^{k+1}-x\|^2-\|x-x^k\|^2\\
\leq &2\alpha [L(x, y^{k+1}, \lambda^k) - L(x^{k+1}, y^{k+1}, \lambda^k)]+2\alpha e^{x,k}_y(x)+\alpha^2 \|H_y^{x,k+1}\|^2\\
&+ (1+\alpha \mu_x)((\rho_1^x)^{-1}-1)\|\hat{x}^{k+1}-x^{k+1}\|^2-(1-\alpha(\rho_2^x+\bar{L}_{\phi,x}))\|x^{k}-x^{k+1}\|^2 \\
&+\frac{\alpha}{\rho_2^x}\|\nabla g(x^k) - \nabla g_{\mathcal{S}_k}(x^k) + \nabla_x f(x^k, y^{k+1}) - \nabla_x f_{\mathcal{S}_k}(x^k, y^{k+1}) - v_y^{x,k}\|^2,
\end{align*}
where we use \(-\alpha \mu_x-(1+\alpha \mu_x)(1-(\rho_1^x)^{-1}) \leq(1+\alpha \mu_x)((\rho_1^x)^{-1}-1).\)
Similarly, it follows from \(-\alpha \mu_y-(1+\alpha \mu_y)(1-(\rho_1^y)^{-1}) \leq (1+\alpha \mu_y)((\rho_1^y)^{-1}-1)\) that
\begin{align*}
\tag{C.2b}
\label{eq:important_inequality_y}
&(1+\alpha \mu_y)(1-\rho_1^y)\|y^{k+1}-y\|^2-\|y-y^k\|^2\\
\leq &2\alpha[L(x^k, y^{k+1}, \lambda^k) - L(x^{k}, y, \lambda^k)] +2\alpha e^{y,k}_x(y)+ \alpha^2 \|H_x^{y,k+1}\|^2\\
&-(1-\alpha(\bar{L}_{\phi,y}+\rho_2^y)) \|y^k - y^{k+1}\|^2+(1+\alpha \mu_y)((\rho_1^y)^{-1}-1)\|\hat{y}^{k+1}-y^{k+1}\|^2 \\
&+ \frac{\alpha}{\rho_2^y}\|\nabla h(y^k) - \nabla h_{\mathcal{S}_k}(y^k) -\nabla_y f(x^k, y^{k}) + \nabla_y f_{\mathcal{S}_k}(x^k, y^{k}) +v^{y,k}_x\|^2.
\end{align*}
Moreover, by Lipschitz smoothness,~\eqref{eq:Hxky-definition} and \hyperref[prop:control_inexact]{proposition \ref{prop:control_inexact}}, it holds that\begin{small}
\begin{align*}
\quad\|H^{x,k+1}_y\|
 &= \|(\mathcal{A}_{\mathcal{S}_k}^y)^\top (\xi_y^{x,k+1} - \hat{\xi}_y^{x,k+1}) +\bigl( \nabla g_{\mathcal{S}_k}(\hat{x}^{k+1}) - \nabla g_{\mathcal{S}_k}(x^{k+1}) \bigr) \\
&\quad +\bigl( \nabla_x f_{\mathcal{S}_k}(\hat{x}^{k+1}, \hat{y}^{k+1}) - \nabla_x f_{\mathcal{S}_k}(x^{k+1}, y^{k+1}) \bigr) \|\\
&\leq \frac{\|\mathcal{A}_{\mathcal{S}_k}^x\|\varepsilon_{\text{sub}}}{\mu_x^*} +L_g\frac{\alpha}{\mu_x^* b} \varepsilon_{\text{sub}} +\frac{\alpha}{b} L_f \varepsilon_{\text{sub}} \left( \frac{1}{\mu_x^*} + \frac{1}{\mu_y^*} \right),
\end{align*}
which implies that
\begin{equation*}
    \label{leq:H_x}
    \tag{H.2a}
\|H^{x,k+1}_y\|\leq \left[\frac{\|\mathcal{A}_{\mathcal{S}_k}^x\|}{\mu^*_x}+
L_g\frac{\alpha}{\mu_x^* b}+\frac{\alpha}{b} L_f(\frac{1}{\mu_x^*}+\frac{1}{\mu_y^*})\right]
\varepsilon_{\text{sub}}
\end{equation*}
\begin{equation}
    \label{leq:H_y}
    \tag{H.2b}
\text{and\ }\|H^{y,k+1}_x\|\leq \left[\frac{\|\mathcal{A}_{\mathcal{S}_k}^y\|}{\mu^*_y}+L_h\frac{\alpha}{\mu_y^* b}+
\frac{\alpha}{b} L_f(\frac{1}{\mu_x^*}+\frac{1}{\mu_y^*})\right]\varepsilon_{\text{sub}}.
\end{equation}\end{small}
By the strong convexity and strong concavity properties, we have
$$
L(x^*, y^{k+1}, \lambda^k) - L(x^{k+1}, y^{k+1}, \lambda^k) \leq -\frac{\mu_x}{2} \| x^{k+1} - x^* \|^2$$
$$L(x^k, y^{k+1}, \lambda^k) - L(x^k, y^*, \lambda^k) \leq -\frac{\mu_y}{2} \| y^{k+1} - y^* \|^2.
$$
By \hyperref[thm:variance_reduce_1]{Theorem~\ref{thm:variance_reduce_1}} with $x=x^k,\tilde{x}=\tilde{x}^s,y=y^{k+1},\tilde{y}=\tilde{y}^s$, we have 
$$
\mathbb{E} \left\| \nabla g(x^k) - \nabla g_{s_k}(x^k) + \nabla_x f(x^k, y^{k+1}) - \nabla_x f_{s_k}(x^k, y^{k+1}) - v_y^{x,k} \right\|^2 
\leq \frac{\bar{L}_{\phi, x}^2}{b} [\| x^k - \tilde{x}^s \|^2+\|y^{k+1}-\tilde{y}^s\|^2],
$$
and with $x=x^k,\tilde{x}=\tilde{x}^s,y=y^{k},\tilde{y}=\tilde{y}^s$,  we have
$$
\mathbb{E} \left\| \nabla h(y^k) - \nabla h_{\mathcal{S}_k}(y^k) -\nabla_y f(x^k, y^{k}) + \nabla_y f_{\mathcal{S}_k}(x^k, y^{k}) +v^{y,k}_x \right\|^2 
\leq \frac{\bar{L}_{\phi, y}^2}{b} [\| x^k - \tilde{x}^s \|^2+\|y^{k}-\tilde{y}^s\|^2].
$$
Moreover,  
\(
\mathbb{E}[e_k^x(x^*)] = 0 \text{\ and\ }\mathbb{E}[e_k^y(y^*)] = 0.
\)  
In addition, by definition and due to the Lipschitz continuity of $F^x_{y,nat},F^y_{x,nat}$. We
obtain
\(\varepsilon_k\leq \delta_s\|F^y_{x,nat}(\tilde{y}^s)\|\leq  (2+L_f+L_h)\delta_s\|\tilde{y}^s-y^*\| ,
\varepsilon_k\leq \delta_s\|F^x_{y,nat}(\tilde{x}^s)\|\leq  (2+L_f+L_g)\delta_s\|\tilde{x}^s-x^*\| .\)
We now choose $x = x^*$, combining our previous results, and taking expectation, it follows
\begin{small} 
\begin{align*}    
    &(1+\alpha \mu_x)(1-\rho_1^x)\E\|x^{k+1}-x^*\|^2-\E\|x^*-x^k\|^2 \\
    \leq{}& -\alpha \mu_x\E \|x^{k+1}-x^*\|^2 + c^x(\alpha)\delta_s^2\|\tilde{x}^s-x^*\|^2 \\
    & +\frac{\alpha\bar{L}^2_{\phi,x}}{\rho^x_2b}\left[\| x^k - \tilde{x}^s \|^2+\|y^{k+1}-\tilde{y}^s\|^2\right]
    -(1-\alpha(\rho_2^x+\bar{L}_{\phi,x}))\|x^{k}-x^{k+1}\|^2.
\end{align*}
\end{small}%
For the $y$ update, we have
\begin{small}
\begin{align*}    
    &(1+\alpha \mu_y)(1-\rho_1^y)\E\|y^{k+1}-y^*\|^2-\E\|y^*-y^k\|^2 \\
    \leq{}& -\alpha \mu_y\E \|y^{k+1}-y^*\|^2 + c^y(\alpha)\delta_s^2\|\tilde{y}^s-y^*\|^2 \\
    & +\frac{\alpha\bar{L}^2_{\phi,y}}{\rho^y_2b}\left[\| x^k - \tilde{x}^s \|^2+\|y^{k}-\tilde{y}^s\|^2\right]
    -(1-\alpha(\rho_2^y+\bar{L}_{\phi,y}))\|y^{k}-y^{k+1}\|^2.
\end{align*}where the coefficient terms $c^x(\alpha)$ and $c^y(\alpha)$ as
\( c^x(\alpha) \coloneqq \alpha^2 \left[\frac{\|\mathcal{A}_{\mathcal{S}_k}^x\|}{\mu^*_x}+ \frac{\alpha L_g}{\mu_x^* b}+\frac{\alpha L_f}{b} \left(\frac{1}{\mu_x^*}+\frac{1}{\mu_y^*}\right)\right]^2 (2+L_f+L_g)^2, 
c^y(\alpha) \coloneqq \alpha^2 \left[\frac{\|\mathcal{A}_{\mathcal{S}_k}^y\|}{\mu^*_y}+ \frac{\alpha L_h}{\mu_y^* b}+\frac{\alpha L_f}{b} \left(\frac{1}{\mu_x^*}+\frac{1}{\mu_y^*}\right)\right]^2 (2+L_f+L_h)^2.
\)\end{small}
Then, we have
\begin{small}
\begin{align*}
    &(1+2\alpha \mu_x-\rho_1^x-\alpha \rho^x_1\mu_x)\E\|x^{k+1}-x^*\|^2\\
\leq&\E\|x^*-x^k\|^2+c^x(\alpha)\delta_s^2\|\tilde{x}^s-x^*\|^2-(1-\alpha(\rho_2^x+\bar{L}_{\phi,x}))\|x^{k}-x^{k+1}\|^2+\frac{\alpha\bar{L}^2_{\phi,x}}{\rho^x_2b}[\| x^k - \tilde{x}^s \|^2+\|y^{k+1}-\tilde{y}^s\|^2] \tag{5.10a}
    \label{x:important_inequality}\\
 &(1+2\alpha \mu_y-\rho_1^y-\alpha \rho^y_1\mu_y)\E\|y^{k+1}-y^*\|^2\\
\leq&\E\|y^*-y^k\|^2+c^y(\alpha)\delta_s^2\|\tilde{y}^s-y^*\|^2-(1-\alpha(\rho_2^y+\bar{L}_{\phi,y}))\|y^{k}-y^{k+1}\|^2+\frac{\alpha\bar{L}^2_{\phi,y}}{\rho^y_2b}[\| x^k - \tilde{x}^s \|^2+\|y^{k}-\tilde{y}^s\|^2]\tag{5.10b}
    \label{y:important_inequality}
\end{align*}\end{small}
For $ k<m$, $\tilde{x}^s=x^0$, we have
$$
\E\|x^k-\tilde{x}^s\|^2=\E\|\sum_{i=0}^{k-1}(x^{i+1}-x^i)\|^2
\leq k\sum_{i=0}^{k-1}\E\|x^{i+1}-x^{i}\|^2\leq k\sum_{i=0}^{m-2}\E\|x^{i+1}-x^i\|^2.$$
Summing this estimate for $k = 0,\cdots,m-1$ gives that
$$\sum_{k=0}^{m-1}\E\|x^k-\tilde{x}^s\|^2\leq \sum_{k=0}^{m-1} k \sum_{i=0}^{m-2}\E\|x^{i+1}-x^i\|^2\leq \frac{m(m-1)}{2}\sum_{i=0}^{m-2}\E\|x^{i+1}-x^i\|^2.$$
Similarly, we have
\(\sum_{k=0}^{m-1}\E\|y^k-\tilde{y}^s\|^2\leq  \frac{m(m-1)}{2}\sum_{i=0}^{m-2}\E\|y^{i+1}-y^i\|^2\).
We now suppose that $\rho^x_1$ is chosen such that $2\alpha \mu_x > \rho_1^x(1+\alpha \mu_x)$,
 $\rho^y_1$ is chosen such that $2\alpha \mu_y > \rho_1^y(1+\alpha \mu_y)$.
 We set $\rho_1=\max\{\rho_1^x,\rho_1^y\},$  we set
 \(\rho_1 <\min\left\{\frac{2\alpha \mu_x}{1+\alpha \mu_x},\frac{2\alpha \mu_y}{1+\alpha \mu_y} \right\}\).
  Then, summing the \eqref{x:important_inequality} estimate for $k = 0,\cdots,m-1$ this implies
\begin{align*}
    &(1+2\alpha \mu_x-\rho_1^x-\alpha \rho^x_1\mu_x)\E\|\tilde{x}^{s+1}-x^*\|^2\\
\leq&\E\|x^*-\tilde{x}^s\|^2+c^x(\alpha)\delta_s^2\E\|\tilde{x}^s-x^*\|^2-(1-\alpha(\rho_2^x+\bar{L}_{\phi,x}))\sum_{k=0}^{m-1}\|x^{k}-x^{k+1}\|^2\\
&+\frac{\alpha\bar{L}^2_{\phi,x}}{\rho^x_2b}\sum_{k=0}^{m-1}[\| x^k - \tilde{x}^s \|^2+\|y^{k+1}-\tilde{y}^s\|^2] \\
\leq&(1+c^x(\alpha)\delta_s^2)\E\|\tilde{x}^s-x^*\|^2-(1-\alpha(\rho_2^x+\bar{L}_{\phi,x}+\frac{m(m-1)\bar{L}^2_{\phi,x}}{2\rho^x_2b}))\sum_{k=0}^{m-1}\|x^{k}-x^{k+1}\|^2\\
&+\frac{m(m-1)\alpha\bar{L}^2_{\phi,x}}{2\rho^y_2b}\sum_{k=0}^{m-1}\|y^{k+1}-y^k\|^2
\end{align*}
and
\begin{align*}
&(1+2\alpha \mu_y-\rho_1^y-\alpha \rho^y_1\mu_y)\E\|\tilde{y}^{s+1}-y^*\|^2\\
\leq&(1+c^y(\alpha)\delta_s^2)\E\|\tilde{y}^s-x^*\|^2-(1-\alpha(\rho_2^y+\bar{L}_{\phi,y}+\frac{m(m-1)\bar{L}^2_{\phi,y}}{2\rho^y_2b}))\sum_{k=0}^{m-1}\|y^{k}-y^{k+1}\|^2\\
&+\frac{m(m-1)\alpha\bar{L}^2_{\phi,y}}{2\rho^y_2b}\sum_{k=0}^{m-1}\|x^{k+1}-x^k\|^2.
\end{align*}
Choosing \(\rho_2^x=\frac{\bar{L}_{\phi,x}\sqrt{m(m-1)}}{\sqrt{2b}}, \rho_2^y=\frac{\bar{L}_{\phi,y}\sqrt{m(m-1)}}{\sqrt{2b}}\),and set $\rho_1=\delta_s$.
We take the minimum value between \((1 + 2\alpha \mu_x - \rho_1^x - \alpha \rho_1^x \mu_x)\) and 
\((1 + 2\alpha \mu_y - \rho_1^y - \alpha \rho_1^y \mu_y)\). We set $\mu_{min}:=\min\{\mu_x,\mu_y\},\mu_{max}:=\max\{\mu_x,\mu_y\}$,
 obviously , we have
\begin{equation*}
    (1 + 2\alpha \mu_{min} - \rho_1 - \alpha \rho_1 \mu_{max}) \leq \min\big\{(1 + 2\alpha \mu_x - \rho_1^x - \alpha \rho_1^x \mu_x),\ (1 + 2\alpha \mu_y - \rho_1^y - \alpha \rho_1^y \mu_y)\big\}.
\end{equation*}
Then we take the maximum value between \(c_x(\alpha)\) and \(c_y(\alpha)\). Similarly, we set
\(
    c^z(\alpha) := \max\big\{c^x(\alpha),\ c^y(\alpha)\big\}.
\)
Then,we have
\begin{align*}
    & (1 + 2\alpha \mu_{min} - \rho_1 - \alpha \rho_1 \mu_{max}) \E\|(\tilde{x}^{s+1},\tilde{y}^{s+1})-(x^*,y^*)\|^2\\
 \leq&(1+2\alpha \mu_x-\rho_1^x-\alpha \rho^x_1\mu_x)\E\|\tilde{x}^{s+1}-x^*\|^2+(1+2\alpha \mu_y-\rho_1^y-\alpha \rho^y_1\mu_y)\E\|\tilde{y}^{s+1}-y^*\|^2\\
\leq& (1+c^x(\alpha)\delta_s^2)\E\|\tilde{x}^s-x^*\|^2+(1+c^y(\alpha)\delta_s^2)\E\|\tilde{y}^s-y^*\|^2\\
&\quad+\frac{m(m-1)\alpha\bar{L}^2_{\phi,y}}{2\rho^y_2b}\sum_{k=0}^{m-1}\|x^{k+1}-x^k\|^2+\frac{m(m-1)\alpha\bar{L}^2_{\phi,x}}{2\rho^y_2b}\sum_{k=0}^{m-1}\|y^{k+1}-y^k\|^2\\
&\quad-(1-\alpha(\rho_2^y+\bar{L}_{\phi,y}+\frac{m(m-1)\bar{L}^2_{\phi,y}}{2\rho^y_2b}))\sum_{k=0}^{m-1}\|y^{k}-y^{k+1}\|^2\\
&\quad-(1-\alpha(\rho_2^x+\bar{L}_{\phi,x}+\frac{m(m-1)\bar{L}^2_{\phi,x}}{2\rho^x_2b}))\sum_{k=0}^{m-1}\|x^{k}-x^{k+1}\|^2\\
\leq&(1+c^z(\alpha)\delta_s^2) \E\|(\tilde{x}^{s},\tilde{y}^{s})-(x^*,y^*)\|^2,\\
\end{align*}
where we set the following
\(
\alpha < \min \left\{ 
\left( 
\bar{L}_{\phi,y} 
\right. \right.  + \frac{\sqrt{2m(m-1)}\,\bar{L}_{\phi,y}}{\sqrt{b}} + \frac{\sqrt{m(m-1)}\,\bar{L}_{\phi,x}}{\sqrt{2b}} 
\left. \vphantom{\frac{\sqrt{2m(m-1)}}{\sqrt{b}}} \right)^{-1},\)

\( \left( 
\bar{L}_{\phi,x} 
+ \frac{\sqrt{2m(m-1)}\,\bar{L}_{\phi,x}}{\sqrt{b}} + \frac{\sqrt{m(m-1)}\,\bar{L}_{\phi,y}}{\sqrt{2b}} 
\left. \vphantom{\frac{\sqrt{2m(m-1)}}{\sqrt{b}}} \right)^{-1} 
\right\}.
\)
Now, we make $\rho_1=\delta_s$. Then, we have
\begin{align*}
     & \E\|(\tilde{x}^{s+1},\tilde{y}^{s+1})-(x^*,y^*)\|^2\\
\leq&\left( \frac{1+c^z(\alpha)\delta_s^2}{1 + 2\alpha \mu_{min} - \rho_1 - \alpha \rho_1 \mu_{max}}\right)\E\|(\tilde{x}^{s},\tilde{y}^{s})-(x^*,y^*)\|^2\\
=&\left( 1-\frac{ 2\alpha\mu_{min} }{1 + 2\alpha \mu_{min} - \rho_1 - \alpha \rho_1 \mu_{max}}+\frac{ \rho_1+\alpha \rho_1 \mu_{max}+c^z(\alpha)\delta_s^2}{1 + 2\alpha \mu_{min} - \rho_1 - \alpha \rho_1 \mu_{max}}\right)\E\|(\tilde{x}^{s},\tilde{y}^{s})-(x^*,y^*)\|^2\\
\leq&\left( 1-\frac{ 2\alpha\mu_{min} }{1 + 2\alpha \mu_{min}}+O(\delta_s)\right)\E\|(\tilde{x}^{s},\tilde{y}^{s})-(x^*,y^*)\|^2, \quad \text{as}\quad s \to \infty.
\end{align*}
\noindent The last inequality follows from the following assumption
\(\rho_1 <\min \left\{ \frac{1+2\alpha \mu_{min}}{1+\alpha \mu_{max}} ,\frac{2\alpha \mu_x}{1+\alpha \mu_x},\frac{2\alpha \mu_y}{1+\alpha \mu_y}      \right\}\)
 This proves q-linear convergence of \( \{(x^s,y^s)\} \) to \( (x^*,y^*) \) in expectation.
Now,we have known
\begin{align*}
    &\E\|(\tilde{x}^{s+1},\tilde{y}^{s+1})-(x^*,y^*)\|^2\\
    \leq&\left( 1-\frac{ 2\alpha\mu_{min} }{1 + 2\alpha \mu_{min}}+O(\delta_s)\right)\E\|(\tilde{x}^{s},\tilde{y}^{s})-(x^*,y^*)\|^2, \quad \text{as}\quad s \to \infty.
\end{align*}

Next, we prove the R-linear convergence of $\lambda$. By Assumption \ref{tolerance_bound} (with $L_{\lambda} = \max(\norm{A}, \norm{B})$) and the update rule (where $\tilde{\lambda}^s = \lambda^{s-1,m}$), the dual error is bounded by the primal error
\begin{equation}\tag{5.11}
\label{eq:lambda_bound_sq}
\norm{\tilde{\lambda}^s - \lambda^{*}}^2 \leq \kappa^2 L_{\lambda}^2 \norm{(\tilde{x}^s, \tilde{y}^s) - (x^*, y^*)}^2.
\end{equation}
Recalling the linear convergence of the primal variables established previously, we have
$$\Expect \norm{(\tilde{x}^{s}, \tilde{y}^{s}) - (x^*, y^*)}^2 \leq C_0 \rho^s$$ for some $C_0 > 0$ and $\rho \in (0,1)$. Taking the expectation of \eqref{eq:lambda_bound_sq} and substituting this bound yields:
\begin{equation}\tag{5.12}
\label{eq:lambda_geometric}
\Expect \norm{\tilde{\lambda}^s - \lambda^{*}}^2 \leq \kappa^2 L_{\lambda}^2 \Expect \norm{(\tilde{x}^s, \tilde{y}^s) - (x^*, y^*)}^2 \leq \underbrace{(\kappa L_{\lambda})^2 C_0}_{:=C_1} \rho^s.
\end{equation}
Since $\rho < 1$, \eqref{eq:lambda_geometric} demonstrates the R-linear convergence of $\{\tilde{\lambda}^s\}$ to $\lambda^*$ in expectation.
\end{proof}         

\begin{assump}\label{full_row}
Suppose that \( A \) and \( B \) satisfy that \( [A\ B] \) is of full row rank.
\end{assump}

It follows from \hyperref[full_row]{Assumption \ref{full_row}} that \( AA^T + BB^T \) is positively definite. 
\hyperref[full_row]{Assumption \ref{full_row}} can be seen as an extension of the linear independent constraint qualification (LICQ) in nonlinear programming, 
which is a general assumption to ensure existence and uniqueness of Lagrange multipliers (e.g., \cite{bonnans2013perturbation,kyparisis1985uniqueness, wachsmuth2013licq}). 
We give the expression of the projection of a point in \( \R^n \times \R^m \) onto \( C \) in the following lemma.

\begin{Lemma}
Let \hyperref[full_row]{Assumption \ref{full_row}} be satisfied. Then for any \( (\tilde{x}, \tilde{y}) \), which is output by Algorithm \ref{alg:SNmMSPP}, its projection on \( C \), denoted by \( \Pi_C(\tilde{x}, \tilde{y}) \), is given by
\[
\Pi_C(\tilde{x}, \tilde{y}) = (\tilde{x}, \tilde{y}) - \bigl(A^T \tilde{\zeta}, B^T \tilde{\zeta}\bigr)
\]
with \( \tilde{\zeta} = (AA^T + BB^T)^{-1}(A\tilde{x} + B\tilde{y} + c) \).
\end{Lemma}
\begin{rem}
The projection mechanism ensures iterates remain near-feasible despite stochastic noise, enabling efficient recovery from temporary constraint violations without compromising convergence.
\end{rem}

\section{Numerical Experiments}\label{sec:num_experiments}

\subsection{Adversarial attacks in network flow problems}

This section evaluates attack strategies (SNmMSPP, MGD, Random, Max capacity, Greedy) on random networks generated via the Erdos-Renyi model (parameter $p$) with $n$ nodes. Edge capacities $p_{ij}$ and base costs $w_{ij}$ are drawn uniformly from $[1,2]$. To model uncertainty in real-world systems \cite{fu2019network,salmeron2004analysis}, we generate $M=2000$ cost samples $w_{ij}^m \sim \mathcal{N}(w_{ij}, \sigma^2)$. The sink demand $r_t$ is set to 50\% of the maximum flow.

The adversary aims to maximize the expected cost by solving the following regularized SAA problem (with $\eta_y=10^{-5}$)
\begin{equation}
\label{eq:experiment_problem}
\begin{aligned}
\max_{\begin{subarray}{c}\mathbf{0} \leq \mathbf{y} \leq \mathbf{p} \\ 
\sum_{(i,j) \in E} y_{ij} = b\end{subarray}} 
&\min_{\begin{subarray}{c}\mathbf{0} \leq \mathbf{x} \leq \mathbf{p} \\ 
\sum_{(i,i) \in E} x_{ii} = r_{t}\end{subarray}} 
\frac{1}{M}\sum_{m=1}^{M}\sum_{(i,j) \in E} w_{ij}^m \cdot (x_{ij} + y_{ij}) \cdot x_{ij} - \frac{\eta_y}{2} \|\mathbf{y}\|^2 \\
& \text{s.t. } \mathbf{x} + \mathbf{y} \leq \mathbf{p} \\
& \quad \sum_{(i,j) \in E} x_{ij} - \sum_{(j,k) \in E} x_{jk} = 0, \quad \forall j \in V \setminus \{s,t\}.
\end{aligned}
\tag{6.2}
\end{equation}
Here, the unit cost is linear in total flow, and the objective maximizes the network owner's expected cost. 

For numerical implementation, we convert the constraints (flow conservation, capacity $\mathbf{x} + \mathbf{y} + \mathbf{z} = \mathbf{p}$, and budget) into a linear system $Ax + By + Cz = b_{\text{eq}}$ by introducing slack variables $z$. Based on these matrices, we compute the projection matrix $M_{\text{proj}} = (AA^T + BB^T + CC^T)^{-1}$ to facilitate the projection operations.

We compare two solution approaches with optimized parameters
\begin{itemize}    
\item \textbf{Deterministic Baselines:} We replace stochastic costs with averages $\bar{w}_{ij} = \frac{1}{M}\sum w_{ij}^m$ and apply Multiplier Gradient Descent (MGD) (with $T=100$ outer iterations, $K=5$ inner steps, step sizes 0.5) and three heuristic strategies
    \textit{Random attack}: Randomly generates $\mathbf{y}$ satisfying the budget.
       \textit{Max capacity attack}: Prioritizes attacking edges with the largest capacities.
        \textit{Greedy attack}: Prioritizes edges with the lowest cost coefficients \cite{tsaknakis2023minimax}.
   
\item \textbf{Stochastic Algorithm (SNmMSPP):} This method handles sampling uncertainty via mini-batches. We set outer iterations $S=200$, inner iterations $M=5$, step size $\alpha=0.002$, batch size $b_0=10$, and $\varepsilon_{\text{sub}}=10^{-10}$. For the Semismooth Newton step (\hyperref[alg:ssn_generic]{Algorithm \ref{alg:ssn_generic}}), we use $\hat{\gamma} = 0.4$, $\eta = 10^{-7}$, $\rho = 0.99$, $\tau = 0.1$, $\tau_1 = 0.01$, and $\tau_2 = 10^{-6}$.
\end{itemize}

Performance is evaluated using the relative cost increase $\rho = \frac{q_{\text{tot}}(\mathbf{x}_{\text{att}})-q_{\text{tot}}(\mathbf{x}_{\text{cl}})}{q_{\text{tot}}(\mathbf{x}_{\text{cl}})}$, where $\mathbf{x}_{\text{cl}}$ and $\mathbf{x}_{\text{att}}$ denote the minimum cost flow assignments before and after the attack. Results are averaged over 15 independent trials per budget level.

\begin{figure}[H]
    \centering
    \setlength{\abovecaptionskip}{3pt} 
    \setlength{\belowcaptionskip}{0pt}
    
    \begin{minipage}[t]{0.49\textwidth} 
        \centering
        \includegraphics[width=\linewidth]{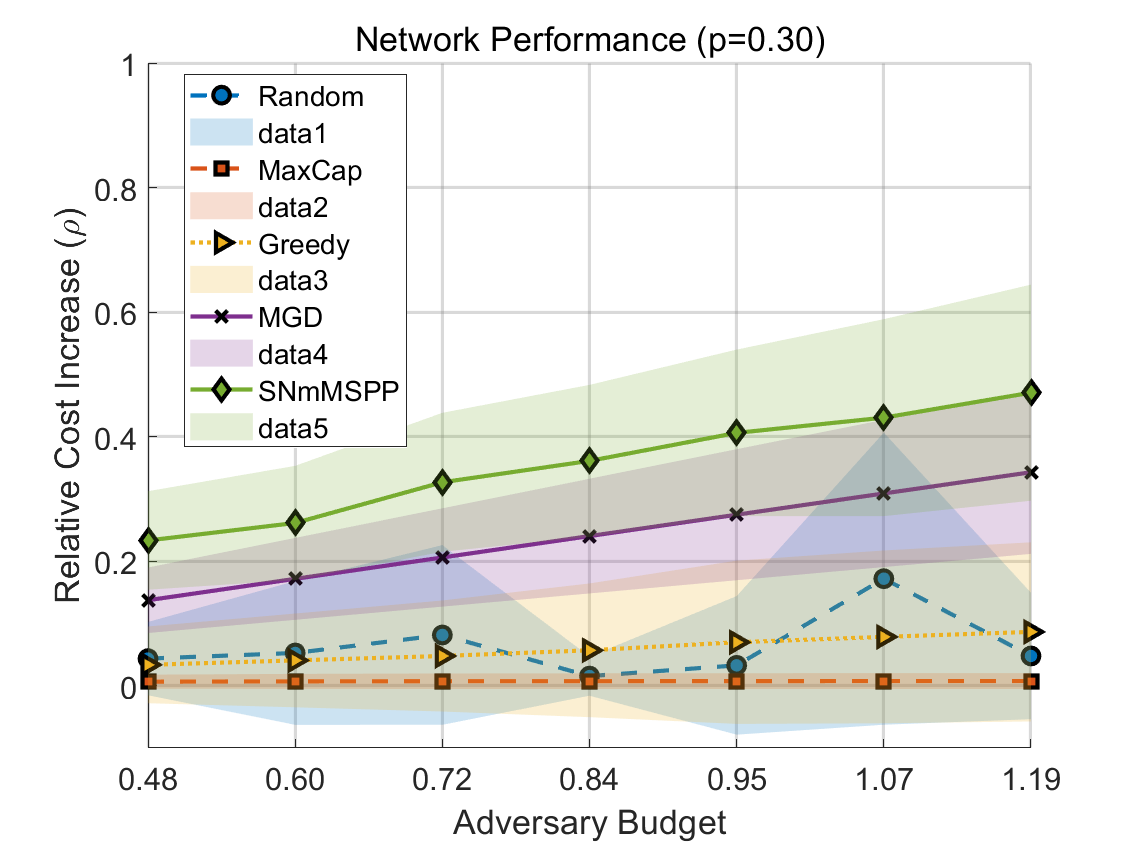}
        \subcaption{p=0.3, n=10, $\sigma=0.001$}
    \end{minipage}%
    \hfill
    \begin{minipage}[t]{0.49\textwidth}
        \centering
        \includegraphics[width=\linewidth]{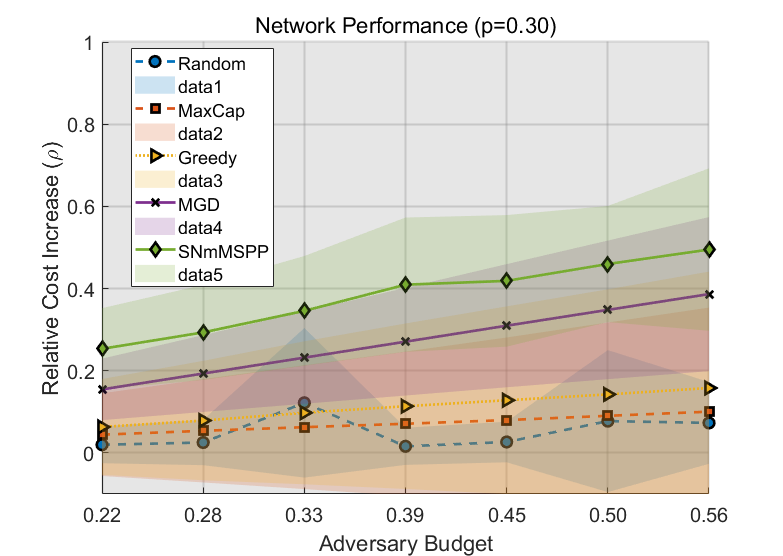}
        \subcaption{p=0.3, n=10, $\sigma=0.01$}
    \end{minipage}
    
    \vspace{4pt} 
    
    \begin{minipage}[t]{0.49\textwidth}
        \centering
        \includegraphics[width=\linewidth]{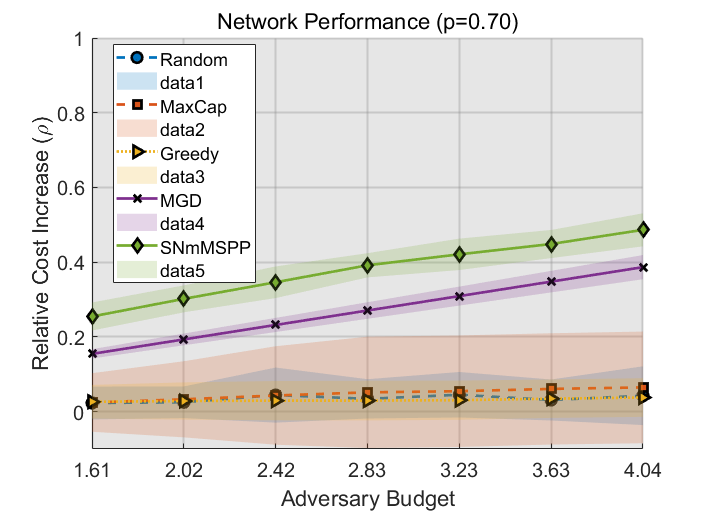}
        \subcaption{p=0.7, n=10, $\sigma=0.001$}
    \end{minipage}%
    \hfill
    \begin{minipage}[t]{0.49\textwidth}
        \centering
        \includegraphics[width=\linewidth]{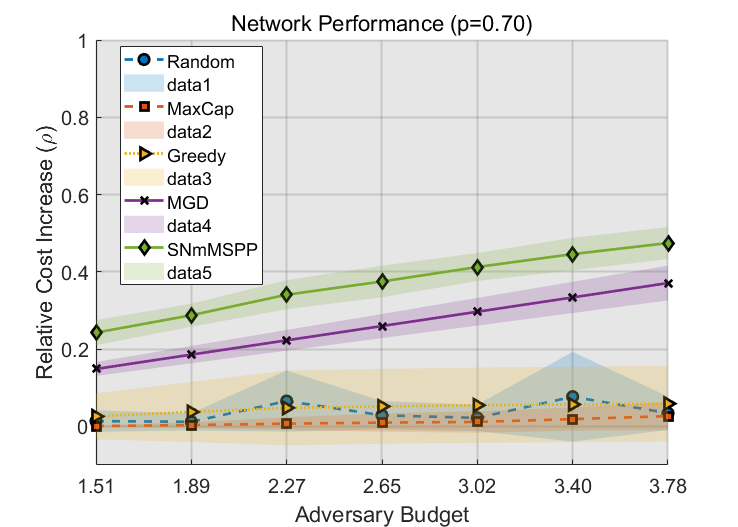}
        \subcaption{p=0.7, n=10, $\sigma=0.01$}
    \end{minipage}
    
    \caption{Comparison of network flow performance under different parameter settings}
    \label{fig:total}
\end{figure}

\subsection{ Linear regression}

In this section, we consider the well-known linear regression problem \cite{du2019linear} with joint linear constraints as follows 
\begin{equation} \tag{6.1}
\begin{aligned}
\min_{x \in \R^n} \max_{y \in \R^m} \quad & f(x, y) = \frac{1}{m} \left[ -\frac{1}{2} \| y \|^2 
- b^T y + \frac{1}{N}\sum_{i=1}^N y^T K_i x \right] + \frac{\lambda}{2} \| x \|^2 \\
\text{subject to} \quad & A x + B y + c = 0_p,
\end{aligned}
\end{equation}  
where we generate a comprehensive dataset of 500,000 random instances for experimental evaluation. The matrix dimensions are set to $n = 100$, $m = 100$, and $p = 50$, with all entries of matrices $K_i \in \R^{m \times n}$, $A \in \R^{p \times n}$, and $B \in \R^{p \times m}$ drawn from a Gaussian distribution $\mathcal{N}(0, \sigma)$ with mean $\mu = 0$ and standard deviation $\sigma = 0.01$. In the following experiments, let $n = m$, $b = 0$, $c = 0$, and $\lambda = 1/m$.

For this experiments, we set the number of outer iterations $S$ in \hyperref[alg:SNmMSPP]{Algorithm \ref{alg:SNmMSPP}} to $S= 30$. For \hyperref[alg:ssn_generic]{Algorithm \ref{alg:ssn_generic}}, we use $\hat{\gamma} = 0.4$, $\eta = 10^{-7}$, $\rho = 0.9$, $\tau = 0.1$, $\tau_1 = 0.01$, $\tau_2 = 1 \times 10^{-6}$ and terminate if $\|\nabla\mathcal{I}(\xi^j)\| \leq 10^{-14}$,i.e.,$\varepsilon_{\text{sub}}=10^{-14}$.

 The relative gradient percentage used in convergence analysis is calculated relative to the initial gradient norm.


\begin{figure}[H]
    \centering
    \begin{subfigure}[b]{0.49\textwidth}
        \centering
        \includegraphics[width=\linewidth]{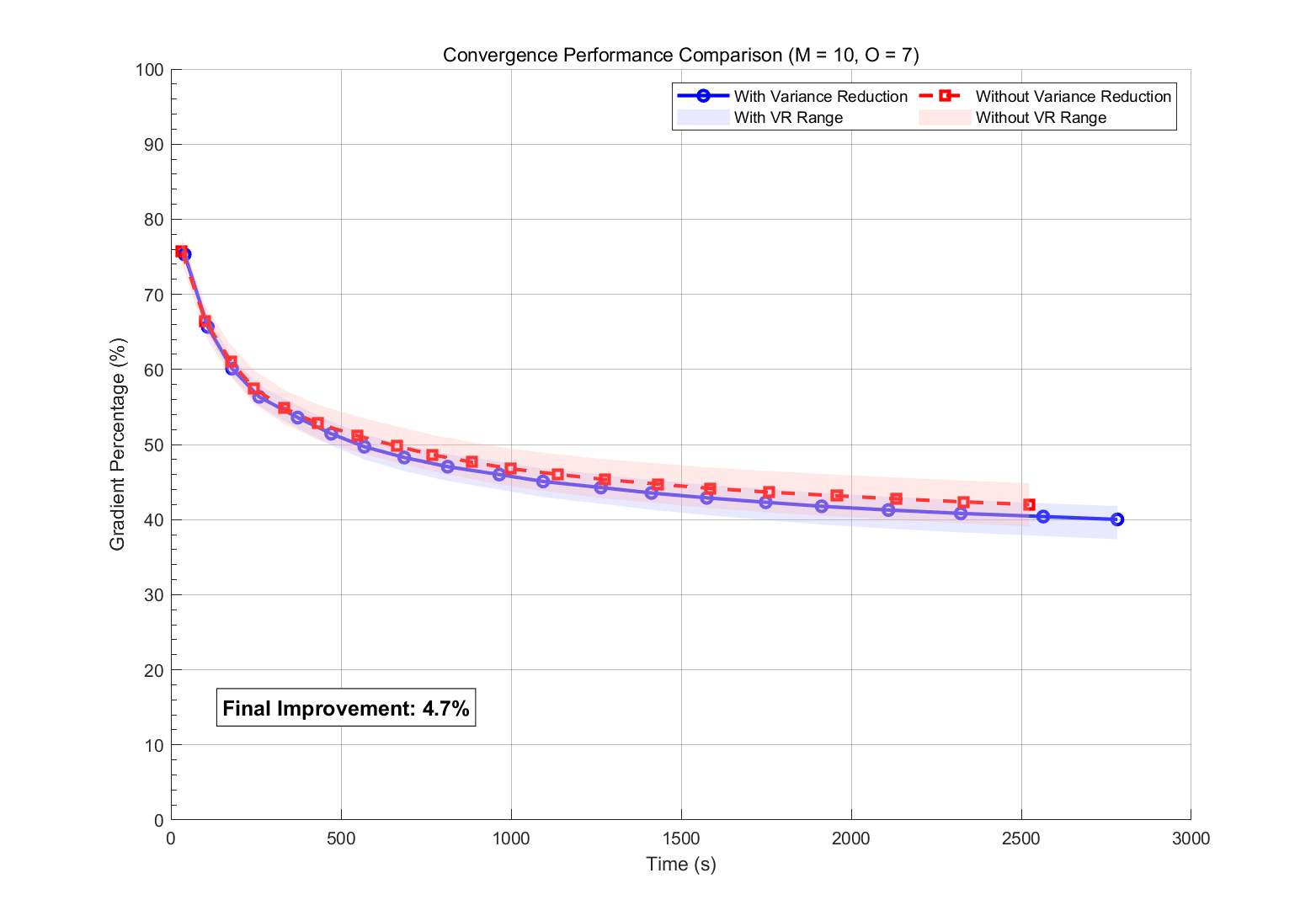}
        \caption{Comparison of step size ($m=10, O=7$)}
        \label{fig:sub1}
    \end{subfigure}
    \hfill
    \begin{subfigure}[b]{0.49\textwidth}
        \centering
        \includegraphics[width=\linewidth]{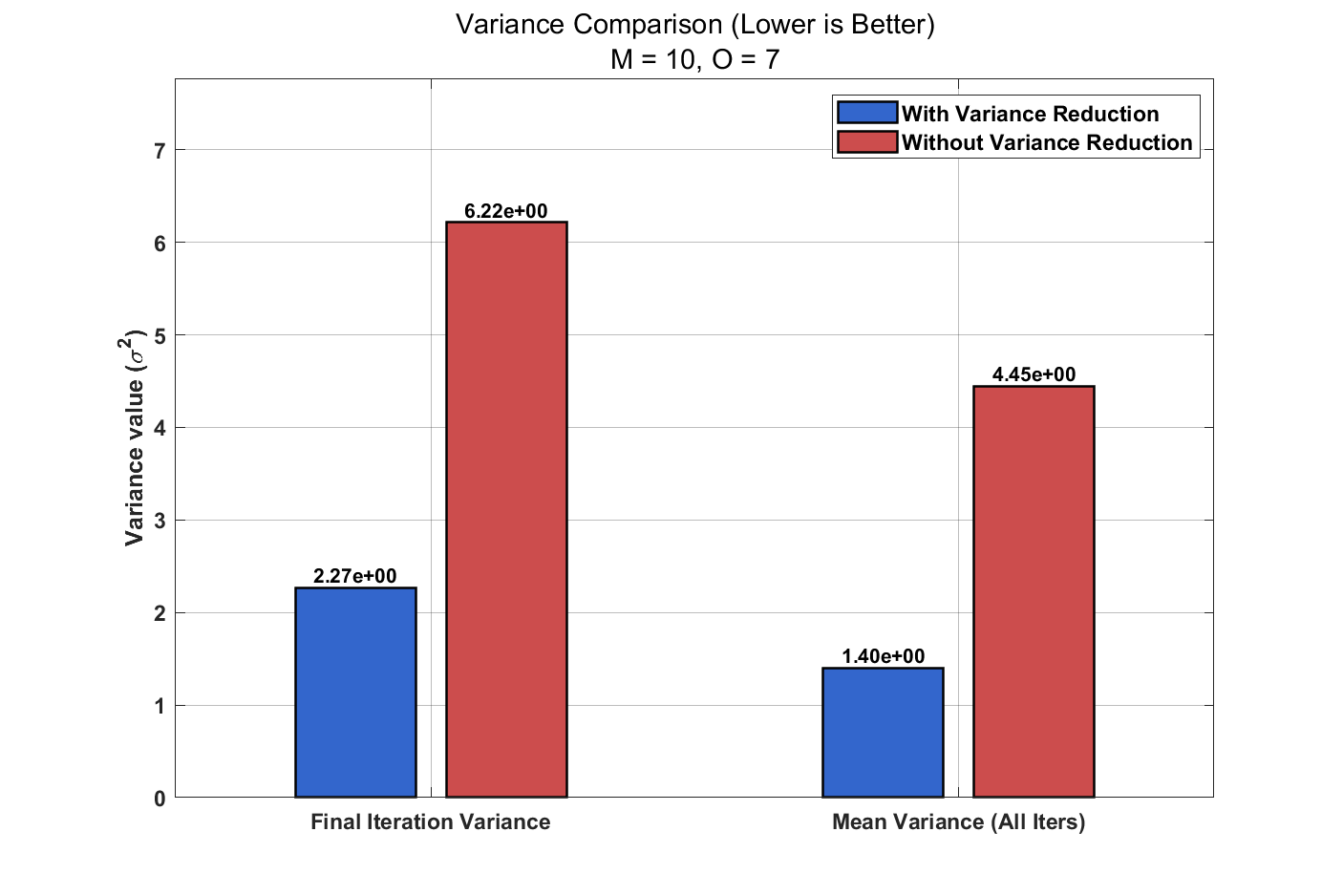}
        \caption{Comparison of inner iterations ($m=10, O=7$)}
        \label{fig:sub2}
    \end{subfigure}
    
    \caption{Performance comparison under different settings. Note that $O$ represents the number of repeated experiments.}
    \label{fig:performance_comparison}
\end{figure}

\begin{figure}[H]
    \centering
    \begin{subfigure}[b]{0.49\textwidth}
        \centering
        \includegraphics[width=\linewidth]{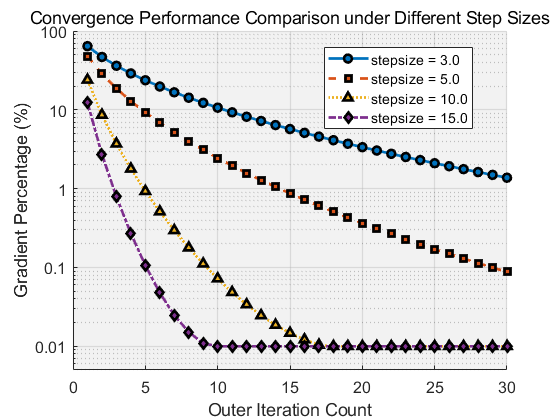}
        \caption{Comparison of step size ($m=15, b_0=5$)}
        \label{fig:sub3}
    \end{subfigure}
    \hfill
    \begin{subfigure}[b]{0.49\textwidth}
        \centering
        \includegraphics[width=\linewidth]{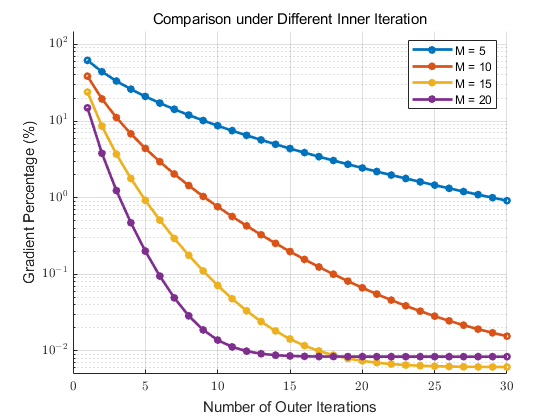}
        \caption{Comparison of inner iterations ($\alpha=5, b_0=5$)}
        \label{fig:sub4}
    \end{subfigure}
    
    \caption{Performance comparison under different settings . It is observed that excessive inner iterations do not strictly imply better overall efficiency due to the trade-off between calculation precision and computational cost.}
\end{figure}

\bibliographystyle{amsplain}
\bibliography{mybib} 
\end{document}